\documentclass{article}
\usepackage{amsmath}
\usepackage{amsthm}
\usepackage{mathrsfs}
\usepackage{latexsym}
\usepackage{amssymb}
\usepackage{amscd}
\usepackage[dvips]{graphics}
\usepackage[all,cmtip]{xy}
\usepackage[margin=1.1in]{geometry}

\usepackage{authblk}

\DeclareSymbolFont{AMSb}{U}{msb}{m}{n}
\DeclareMathSymbol{\N}{\mathbin}{AMSb}{"4E}
\DeclareMathSymbol{\Z}{\mathbin}{AMSb}{"5A}
\DeclareMathSymbol{\R}{\mathbin}{AMSb}{"52}
\DeclareMathSymbol{\Q}{\mathbin}{AMSb}{"51}
\DeclareMathSymbol{\I}{\mathbin}{AMSb}{"49}
\DeclareMathSymbol{\C}{\mathbin}{AMSb}{"43}

\DeclareFontFamily{U}{mathx}{\hyphenchar\font45}
\DeclareFontShape{U}{mathx}{m}{n}{
      <5> <6> <7> <8> <9> <10>
      <10.95> <12> <14.4> <17.28> <20.74> <24.88>
      mathx10
      }{}
\DeclareSymbolFont{mathx}{U}{mathx}{m}{n}
\DeclareFontSubstitution{U}{mathx}{m}{n}
\DeclareMathAccent{\widecheck}{0}{mathx}{"71}
\DeclareMathAccent{\wideparen}{0}{mathx}{"75}

\newcommand{\dbl}{[\hspace{-0.2ex}[}
\newcommand{\dbr}{]\hspace{-0.2ex}]}
\newcommand{\db}[1]{\dbl {#1} \dbr}
\newcommand{\res}[1]{\hspace{-0.6mm}\downarrow_{\hspace{-0.25mm}{#1}}}
\newcommand{\ind}[1]{\hspace{-0.6mm}\uparrow^{\hspace{-0.25mm}{#1}}}

\newcommand{\ctens}{\widehat{\otimes}}

\newcommand{\iso}{\cong}
\newcommand{\invlim}{\underleftarrow{\textnormal{lim}}\,}
\newcommand{\dirlim}{\underrightarrow{\textnormal{lim}}\,}
\newcommand{\ds}{\raisebox{0.5pt}{\,\big|\,}}
\newcommand{\onto}{\twoheadrightarrow}

\newcommand{\norm}[1]{\textnormal{N}_{#1}}

\newcommand{\dash}{\textnormal{-}}
\newcommand{\tn}[1]{\textnormal{#1}}

\newcommand{\xonto}[2][]{%
  \xrightarrow[#1]{#2}\mathrel{\mkern-14mu}\rightarrow
}

\numberwithin{equation}{section}

\title{Infinitely generated pseudocompact modules for finite groups and Weiss' Theorem}

\author[1]{John William MacQuarrie}
\author[2]{Peter Symonds}
\author[3]{Pavel Zalesskii}
\affil[1]{Universidade Federal de Minas Gerais, {{john@mat.ufmg.br}}}
\affil[2]{University of Manchester, {{peter.symonds@manchester.ac.uk}}}
\affil[3]{Universidade de Bras\'ilia, {{pz@mat.unb.br}}}

\begin{document}

\newtheorem{defn}[equation]{Def{i}nition}
\newtheorem{prop}[equation]{Proposition}
\newtheorem{lemma}[equation]{Lemma}
\newtheorem{theorem}[equation]{Theorem}
\newtheorem{corol}[equation]{Corollary}
\newtheorem{question}[equation]{Questions}

\maketitle

\section{Introduction}

In 1988, A.\ Weiss \cite{weiss} proved one of the most beautiful theorems in integral representation theory: it characterizes the finitely generated $\Z_p$-permutation lattices for a finite $p$-group $G$ in terms of  information about the restriction to a normal subgroup $N$ and the action of $G/N$ on the $N$-invariants of the module (cf.\ also \cite[\S 6]{weisscorr}, \cite{Roggencamp}).  In 1993, Weiss generalized his own result to finite extensions of $\Z_p$, yielding the following theorem:

\begin{theorem}[\cite{weisspi}, {\cite[App.\ 1]{Puig}}]\label{Theorem Weiss original}
Let $R$ be a finite extension of $\Z_p$, let $G$ be a finite $p$-group and let $U$ be a finitely generated $RG$-lattice.  Suppose there is a normal subgroup $N$ of $G$  such that
\begin{itemize}
\item the module $U$ restricted to $N$ is a free $RN$-module,
\item the submodule of fixed points $U^N$ is a permutation $RG$-module.
\end{itemize}
Then $U$ itself is a permutation $RG$-module.
\end{theorem}

The importance of this theorem is highlighted by its applications.  Weiss originally proved it in order to show that every finite $p$-subgroup of augmentation-1 units in $RG$ is conjugate to a subgroup of $G$ \cite{weiss}. Puig used it to show that if a block of the group algebra $RF$ of a finite group $F$ is stably Morita equivalent to a nilpotent block then it is nilpotent \cite[Thm 8.2]{Puig}. No alternative proofs of these results are known.
We generalize Theorem \ref{Theorem Weiss original} in two ways.  First, we allow the coefficient ring $R$ to be any complete discrete valuation ring in mixed characteristic (that is, with residue field of prime characteristic $p$ and with field of fractions of characteristic $0$).  The second generalization is more important: we prove the result for all pseudocompact $RG$-lattices, possibly infinitely generated. What is important here is that the rank of the lattice is allowed to be infinite.   Pseudocompact is just the correct generalization of profinite to the case when $R$ is not profinite; we work with such modules because they are better behaved than abstract ones in the infinitely generated case.

The notion of permutation profinite modules was introduced by Mel'nikov in \cite{me}, where he studied their basic properties.  As  he  pointed out, profinite permutation modules are  important for the combinatorial theory of profinite groups and especially of pro-$p$ groups.  A different approach, which complements the approach of Mel'nikov in the study of the basic properties of profinite permutation modules, was introduced by the second author \cite{symondspermcom1}, where these modules were used to develop the cohomology theory of profinite groups along the lines followed for discrete groups, as in the book of Brown \cite{Brown}. We note also that permutation profinite modules appear naturally in Galois Theory \cite[Thm 1.3]{BJN}.

Before stating our main theorem for a finite $p$-group $G$ we define a pseudocompact permutation $RG$-lattice $U$ to be a pseudocompact module that is free over $R$ and having a pointed compact $G$-invariant basis (several alternative definitions are given in detail and compared in Section \ref{subsection permutation modules}). 

\begin{theorem}\label{main}
Let $R$ be a complete discrete valuation ring in mixed characteristic with residue field of characteristic $p$, let $G$ be a finite $p$-group and let $U$ be a pseudocompact $RG$-lattice.  Suppose there is a normal subgroup $N$ of $G$  such that
\begin{itemize}
\item the module $U$ restricted to $N$ is a free $RN$-module,
\item the submodule of fixed points $U^N$ is a permutation $RG$-module.
\end{itemize}
Then $U$ itself is a permutation $RG$-module.
\end{theorem}

This theorem plays a central role in the proof of the pro-$p$ version \cite{PavelVirtuallyFree} of the theorem of Karras, Pietrowski, Solitar, Cohen and Scott \cite{KPS, CohenD, Scott},  which states that a virtually free
group acts on a tree with finite vertex stabilizers. Indeed, in the pro-$p$ case Theorem \ref{main} replaces Stallings' theory of ends, crucial in the proof of the original result.  Recently, the greater generality of coefficient ring has yielded applications of Theorem \ref{main} in the calculation of Picard groups of blocks of finite groups \cite{BolKesLin, Eisele, Livesey}.

One might expect that the generalization to infinite rank lattices could be proved using a simple limit argument, but this does not seem to be the case. Instead, we need to recast the proof in a way that does not depend on expressing a module as a sum of indecomposable modules. In order to do this we look closely at the properties of various classes of infinitely generated modules; these results should be of independent interest.

We make extensive use of relative homological algebra, and in particular of the concepts of covers and precovers of a module by a module with desirable properties.  We prove the existence of a large class of covers of pseudocompact modules (Theorem \ref{covers exist}).  We give an explicit description of the permutation cover of a pseudocompact module for a finite group when $R$ is a complete discrete valuation ring (Theorem \ref{Theorem explicit permutation cover}).  We also generalize a related theorem of Cliff and Weiss (Theorem \ref{ThmCliffWeiss}).

\subsubsection*{Acknowledgements}

The authors thank the anonymous referee, whose thorough reading of the article has significantly improved the exposition in several places.  The first author was partially supported by a CNPq Universal Grant, a UFMG Rec\'em Contratados grant and a FAPEMIG PPM Grant.  The second author was partially supported by an International Academic Fellowship from the Leverhulme Trust.  The third author wishes to acknowledge the financial support of CNPq and FAPDF.

\section{Definitions, terminology and background}

\subsection{Algebras and modules}

Our coefficient ring $R$ will always be a commutative pseudocompact ring.  In later sections we will require further structure on $R$, the main coefficient rings of interest to us being complete discrete valuation rings.

Let $\Lambda$ be a pseudocompact $R$-algebra (we follow the treatments in \cite{Brumer} and \cite{GabrielThesis}).  Examples of particular interest are the completed group algebra $R\db{G}$ of a profinite group $G$, or later the group algebra $RG$ of a finite group $G$.  
We consider the following categories of modules for $\Lambda$:
\begin{itemize}
\item $\Lambda\dash\tn{Mod}^\mathcal{C}$: the category whose objects are pseudocompact left $\Lambda$-modules \cite[\S 1]{Brumer}.  Objects of this category are complete, Hausdorff topological $\Lambda$-modules $U$ having a basis of open neighbourhoods of $0$ consisting of submodules $V$ for which $U/V$ has finite length.  In other words, the category of inverse limits of left $\Lambda$-modules of finite length over $R$.
\item $\Lambda\dash\tn{Mod}^\mathcal{D}$: the category whose objects are those topological $\Lambda$-modules having the discrete topology.  Such modules are precisely the direct limits of left $\Lambda$-modules of finite length over $R$.  We will call such modules \emph{discrete} (they are called ``locally finite'' in \cite{GabrielThesis}).
\item $\Lambda\dash\tn{Mod}^\tn{abs}$: the category of abstract left $\Lambda$-modules
\end{itemize}
Morphisms in $\Lambda\dash\tn{Mod}^\mathcal{C}$ and $\Lambda\dash\tn{Mod}^\mathcal{D}$ are continuous $\Lambda$-module homomorphisms while morphisms in $\Lambda\dash\tn{Mod}^\textnormal{abs}$ are arbitrary $\Lambda$-module homomorphisms.  The corresponding categories of right modules are denoted $\tn{Mod}^\mathcal{C}\dash\Lambda, \tn{Mod}^\mathcal{D}\dash\Lambda$ and $\tn{Mod}^{\tn{abs}}\dash\Lambda$.   We include the decorations $\mathcal{C}, \mathcal{D}, \tn{abs}$ inconsistently, omitting them when the category is clear from the context.  We occasionally write $\tn{Hom}^{\tn{cts}}$ or $\tn{Hom}^{\tn{abs}}$ to make explicit whether we consider continuous or arbitrary homomorphisms.
A \emph{$\Lambda$-lattice} is a $\Lambda$-module in $\Lambda\dash\tn{Mod}^\mathcal{C}, \Lambda\dash\tn{Mod}^\tn{abs}$ or in $\Lambda\dash\tn{Mod}^\mathcal{D}$ that is projective as an $R$-module in the corresponding category.  Note that the way sums and products appear is not quite as one might expect from the abstract case -- for example, free pseudocompact $\Lambda$-modules are direct products of copies of $\Lambda$ (cf.\ \cite[\S 1]{Brumer}).

\medskip

Following \cite{Brumer}, denote by $E_R$ the dualizer of $R$: that is, $E_R$ is the injective hull in $R\dash\tn{Mod}^{\mathcal{D}}$ of the module
$$\bigoplus_{\mathfrak{m}}R/\mathfrak{m},$$
where $\mathfrak{m}$ runs through the maximal ideals of $R$.  The functor sending a module $M$ (in $\Lambda\dash\tn{Mod}^\mathcal{C}$ or $\tn{Mod}^\mathcal{D}\dash\Lambda$) to the module of continuous $R$-module homomorphisms from $M$ to $E_R$ induces a Pontryagin duality between the categories $\Lambda\dash\tn{Mod}^\mathcal{C}$ and $\tn{Mod}^\mathcal{D}\dash\Lambda$ \cite[Prop.2.3]{Brumer} so that any result in one category corresponds to a dual result in the other.  We will at times use this observation without comment.  Given a module $M$ in either the pseudocompact or the discrete category, we denote its Pontryagin dual by $M^*$.

\medskip

Denote by $\Lambda\dash\tn{Proj}$ (resp.\ $\Lambda\dash\tn{Inj}$) the full subcategories of (a given category of) left $\Lambda$-modules having as objects the projective modules (resp.\ injective modules).  The category $\Lambda\dash\tn{Mod}^\mathcal{C}$ has projective covers and the category $\Lambda\dash\tn{Mod}^\mathcal{D}$ has injective hulls \cite[Ch.II, Thm 2]{GabrielThesis}.  An arbitrary product of projective modules in $\Lambda\dash\tn{Mod}^\mathcal{C}$ is projective and an arbitrary sum of injective modules in $\Lambda\dash\tn{Mod}^\mathcal{D}$ is injective (the first statement follows from \cite[Cor.\ 3.3]{Brumer} and the second is its dual).

\medskip

Given a module $M$ in $\Lambda\dash\tn{Mod}^\mathcal{C}$, denote by $\tn{Add}^{\mathcal{C}}(M)$ the full subcategory of $\Lambda\dash\tn{Mod}^\mathcal{C}$ whose objects are summands of products of $M$.  Denote by $\tn{add}^{\mathcal{C}}(M)$ the full subcategory of $\tn{Add}^{\mathcal{C}}(M)$ consisting of summands of finite products of $M$.  Dually, given a module $M$ in $\Lambda\dash\tn{Mod}^\mathcal{D}$ we have the categories $\tn{Add}^{\mathcal{D}}(M)$ of summands of direct sums of $M$ and $\tn{add}^{\mathcal{D}}(M)$ whose objects are summands of finite sums of $M$.  If $M\in \Lambda\dash\tn{Mod}^\mathcal{C}$, the category dual to $\tn{Add}^{\mathcal{C}}(M)$ is $\tn{Add}^{\mathcal{D}}(M^*)$. 

\medskip

Nakayama's Lemma holds in the category $R\dash\tn{Mod}^\mathcal{C}$: if $N$ is a closed submodule of the pseudocompact $R$-module $M$ such that $N+\tn{Rad}(R)M = M$, where $\tn{Rad}(R)$ is the intersection of the maximal ideals of $R$, then $M=N$ \cite[Lem.\ 1.4]{Brumer}.  There is a dual version for $R\dash\tn{Mod}^\mathcal{D}$.

\medskip

It is convenient to be able to describe a $\Lambda$-lattice in terms of $\Lambda$-lattices of finite rank over $R$.  We prove only a special case that we will require later.

\begin{lemma}\label{RG lattice limit of finite rank RG lattices}
Let $G$ be a finite group, $R$ a complete discrete valuation ring and $U$ a pseudocompact $RG$-lattice.  Then $U$ can be expressed as the inverse limit of an inverse system of $RG$-lattices of finite rank with surjective homomorphisms.
\end{lemma}

\begin{proof}
As an $R$-lattice, we can write $U = \prod_{i\in I}R$.  For each cofinite subset $J$ of $I$, let $U_J=\prod_{i\in J}R$ and  $V_J = \bigcap_{g\in G}gU_J$.  Then $V_J$ is an $RG$-submodule of $U$.    The obvious homomorphism 
$$U/V_J \to \bigoplus_{g\in G} U/gU_J$$
is injective, so $U/V_J$ is isomorphic as an $R$-module to a submodule of an $R$-lattice of finite rank, hence it is an $RG$-lattice of finite rank itself, since $R$ is a discrete valuation domain.

Because $V_J \subseteq U_J$ and 
\[ 
U = \invlim_J U/U_J,
\]
we have 
\[ 
U = \invlim_J U/V_J.
\]
\end{proof}

\subsection{Tensor products, homomorphisms, Ext functors}\label{subsection tensor hom and Ext}

The references in this section are sometimes incomplete, insofar as the proofs given are for profinite, rather than pseudocompact modules.  What we mean is that the proof given in the reference also works for pseudocompact modules.

Let $\Lambda, \Gamma$ be pseudocompact $R$-algebras, $M$ a pseudocompact $\Lambda\dash\Gamma$-bimodule (that is, $M$ is a bimodule that is pseudocompact both as a left $\Lambda$-module and as a right $\Gamma$-module) and $N$ a pseudocompact left $\Gamma$-module.  The \emph{completed tensor product} defined in \cite[\S 2]{Brumer} is a pseudocompact $R$-module $M\ctens_{\Gamma}N$ together with a continuous bilinear map $M\times N\to M\ctens_{\Gamma}N$, written $(m,n)\mapsto m\ctens n$ satisfying the condition that $mg\ctens n = m\ctens gn$ for all $m\in M, n\in N, g\in \Gamma$ and universal with respect to this condition.  The completed tensor product inherits naturally the structure of a left $\Lambda$-module with multiplication given on pure tensors by $x(m\ctens n) := xm\ctens n$.  If $M = \invlim M_i, N = \invlim_j N_i$ are expressions of $M,N$ as inverse limits of $\Gamma$-modules of finite $R$-length, then 
$$M\ctens_{\Gamma}N = \invlim_{i,j}M_i\otimes N_j$$
\cite[Lem.\ 5.5.1]{RZ12010}.  The completed tensor product commutes with inverse limits in both variables \cite[Lem.\ 5.5.2]{RZ12010}. It is not, in general, isomorphic to the abstract tensor product, but we do have the following result.

\begin{prop}\label{prop tensors coincide}
If either $M$ or $N$ is finitely presented as a $\Gamma$-module, or if both $M$ and $N$ are finitely generated as $\Gamma$-modules, then the natural map $M\otimes_{\Gamma}N\to M\ctens_{\Gamma}N$ is an isomorphism.
\end{prop}

\begin{proof}
The result for finitely presented modules is \cite[Lem.\ 1.1]{Corob Cook Bieri}.  When both modules are finitely generated, one need only observe that the left-most vertical map in the proof of \cite[Lem.\ 1.1 (iii)]{Corob Cook Bieri} is still surjective, so the same proof works.
\end{proof}

Note that the assertion in \cite[Lem.\ 2.1 (ii)]{Brumer} that the tensor products $\ctens$ and $\otimes$ coincide when only one of the modules involved is finitely generated is false in general \cite[Corol.\ 2.4]{Corob Cook Bieri}. 

\medskip

If $M$ is a topological $\Lambda\dash\Gamma$-bimodule and $N$ is a topological left $\Lambda$-module, $\tn{Hom}_{\Lambda}(M,N)$ denotes the left $\Gamma$-module of continuous $\Lambda$-module homomorphisms from $M$ to $N$ and is given the compact-open topology.  If $M,N$ are pseudocompact, then $\tn{Hom}_{\Lambda}(M,N) = \invlim_i \tn{Hom}_{\Lambda}(M,N_i)$, where $N = \invlim N_i$ and each $N_i$ has finite length over $R$.  The compact-open topology coincides with the topology obtained by declaring each $\tn{Hom}_{\Lambda}(M,N_i)$ to be discrete and giving $\tn{Hom}_{\Lambda}(M,N)$ the inverse limit topology.  In particular, $\tn{Hom}_{\Lambda}(M,N)$ commutes with products in the second variable.  We also have that $\tn{Hom}_{\Lambda}(M,N)$ commutes with finite products in the first variable.  If $M$ is finitely generated as a $\Lambda$-module and $N$ is either pseudocompact or discrete, then every abstract homomorphism from $M$ to $N$ is continuous.  If $M$ is finitely generated and $N$ is pseudocompact then $\tn{Hom}_{\Lambda}(M,N)$ is pseudocompact as an $R$-module.  Furthermore:

\begin{lemma}\label{Lemma E pseudocompact}
If $M$ is a finitely generated pseudocompact $\Lambda$-module, then $E = \tn{End}_{\Lambda}(M)$ is a pseudocompact $R$-algebra.
\end{lemma}

\begin{proof}
Write $\Lambda = \invlim \Lambda/I_i$ with each $\Lambda/I_i$ of finite length over $R$.  Then $M/\tn{Cl}(I_iM)$ has finite length over $R$ (where $\tn{Cl}$ denotes topological closure) and hence so does $\tn{End}_{\Lambda}(M/\tn{Cl}(I_iM))$.  Let $J_i$ denote the two sided ideal $\{\rho\in E\,|\,\rho(M)\leqslant \tn{Cl}(I_iM)\}$ of $E$.  The map $E/J_i\to \tn{End}_{\Lambda}(M/\tn{Cl}(I_iM))$ is injective, so each $E/J_i$ has finite length over $R$.  But $E = \invlim E/J_i$ and hence $E$ is pseudocompact.
\end{proof}

Let $X$ be a pseudocompact $\Lambda\dash\Gamma$-bimodule that is finitely generated as a left $\Lambda$-module.  The functor
$$X\ctens_{\Gamma}- : \Gamma\dash\tn{Mod}^\mathcal{C} \to \Lambda\dash\tn{Mod}^\mathcal{C}$$
is left adjoint to the functor
$$\tn{Hom}_{\Lambda}(X,-): \Lambda\dash\tn{Mod}^\mathcal{C}\to \Gamma\dash\tn{Mod}^\mathcal{C}:$$
given a pseudocompact left $\Gamma$-module $A$ and a pseudocompact left $\Lambda$-module $C = \invlim C_i$ with each $C_i$ of finite length, we have  
\begin{align*}
\tn{Hom}_{\Gamma}(A,\tn{Hom}_{\Lambda}(X,\invlim C_i))
& \iso \tn{Hom}_{\Gamma}(A,\invlim \tn{Hom}_{\Lambda}(X,C_i)) \\
& \iso \invlim\tn{Hom}_{\Gamma}(A,\tn{Hom}_{\Lambda}(X, C_i)) \\
& \iso \invlim\tn{Hom}_{\Lambda}(X\ctens_{\Gamma}A, C_i) \\ 
& \iso \tn{Hom}_{\Lambda}(X\ctens_{\Gamma}A, \invlim C_i),
\end{align*}
where the second to last isomorphism is \cite[Lem.\ 2.4]{Brumer} (though note that the side conventions there are different).  If $\Psi, \Phi$ are other pseudocompact algebras, $A$ is a pseudocompact $\Gamma\dash \Phi$-module and $C$ is a pseudocompact $\Lambda\dash \Psi$-bimodule, then the above isomorphisms are isomorphisms of $\Phi\dash\Psi$-bimodules \cite[Thm 4.2]{BoggiCC}.

If $H$ is a closed subgroup of the profinite group $G$ and $V$ is a pseudocompact $R\db{H}$-module, the induced $R\db{G}$-module $V\ind{G}$ is defined to be
$$V\ind{G} := R\db{G}\ctens_{R\db{H}}V,$$
where $R\db{G}$ is treated as an $R\db{G}\dash R\db{H}$-bimodule in the obvious way.  Given a pseudocompact $R\db{G}$-module $U$, the restriction $U\res{H}$ (that is, $U$ treated as an $R\db{H}$-module) can be expressed as $U\res{H} = \tn{Hom}_{R\db{G}}(R\db{G},U)$, where again $R\db{G}$ is being treated as an $R\db{G}\dash R\db{H}$-bimodule in the obvious way.  It follows from the adjunction above that induction is left adjoint to restriction.  When $H$ is open in $G$, induction is also right adjoint to restriction.  To see this, we need only note that, since $H$ is open, restriction is anyway left adjoint to the coinduction functor $\tn{Hom}_{R\db{H}}(R\db{G},-)$.  But since $R\db{G}$ is finitely presented as an $R\db{H}$-module, every functor involved coincides with its abstract counterpart.  We may thus apply \cite[Lem.\ 6.3.4]{Weibel}.  Let $H,L$ be closed subgroups of the profinite group $G$ with either $H$ or $L$ open.  The Mackey decomposition formula applies.  That is, given a pseudocompact $R\db{H}$-module $V$, we have
$$V\ind{G}\res{L} \iso \bigoplus_{g\in L\backslash G/H}{}^g\!V\res{{}^g\!H\cap L}\ind{L},$$
where $g$ runs through a (finite) set of double coset representatives of $L\backslash G/H$ and ${}^g\!V$ is the $R\db{{}^g\!H}$-module $V$ with action $ghg^{-1}\cdot v = hv$.

\medskip

Given pseudocompact $\Lambda$-modules $M,N$, denote by $\tn{Ext}_{\Lambda}^i(M,N)$ the $i$th right derived functor of $\tn{Hom}_{\Lambda}(-,N)$ applied to $M$.  It can be calculated via a projective resolution for $M$ in the usual way.  Since $\tn{Hom}_{\Lambda}$ commutes with finite products in the first variable and arbitrary products in the second, it follows that $\tn{Ext}^i_{\Lambda}(M,N)$ also commutes with finite products in $M$ and arbitrary products in $N$.  However, it does not commute with arbitrary inverse limits in the second variable.  Since $\Lambda\dash\tn{Mod}^\mathcal{C}$ is abelian and has enough projectives, the functor $\tn{Ext}_{\Lambda}^1$ classifies extensions as one would hope \cite[Vis.\ 3.4.6]{Weibel}.

Dually, in $\Lambda\dash\tn{Mod}^\mathcal{D}$, denote by $\tn{Ext}_{\Lambda}^i(M,N)$ the $i$th right derived functor of $\tn{Hom}_{\Lambda}(M,-)$ applied to $N$.  By Pontryagin duality,
$$\tn{Ext}_{\Lambda\dash\tn{Mod}^{\mathcal{C}}}^i(M,N) \iso \tn{Ext}_{\tn{Mod}^{\mathcal{D}}\dash\Lambda}^i(N^*, M^*).$$

In what follows, when we say that $\Lambda$ is finitely generated over $R$, we mean finitely generated as an $R$-module.  Examples of particular interest are $\Lambda = RG$ for a finite group $G$.

\begin{lemma}
Suppose $\Lambda$ is finitely generated over $R$.  If $M,N$ are pseudocompact $\Lambda$-modules with $M$ finitely generated, and if either $R$ is noetherian or $M$ is a lattice, then
$$\tn{Ext}^i_{\Lambda\dash\tn{Mod}^{\mathcal{C}}}(M,N) \iso \tn{Ext}^i_{\Lambda\dash\tn{Mod}^{\tn{abs}}}(M,N).$$
\end{lemma}

\begin{proof}
Our conditions guarantee that there exists a projective resolution of $M$ in $\Lambda\dash\tn{Mod}^{\mathcal{C}}$ with each module finitely generated.  This is also a projective resolution of $M$ by projective modules in $\Lambda\dash\tn{Mod}^{\tn{abs}}$.  But $\tn{Hom}_{\Lambda\dash\tn{Mod}^{\mathcal{C}}}(P,N) = \tn{Hom}_{\Lambda\dash\tn{Mod}^{\tn{abs}}}(P,N)$ whenever $P$ is finitely generated (that is, every abstract homomorphism $P\to N$ is continuous), and so the groups obtained are the same.
\end{proof}

\begin{lemma}
Suppose that $R$ is noetherian and $\Lambda$ is finitely generated over $R$.  For $M,N\in \Lambda\dash\tn{Mod}^{\mathcal{D}}$ we have
$$\tn{Ext}^i_{\Lambda\dash\tn{Mod}^{\mathcal{D}}}(M,N) \iso \tn{Ext}^i_{\Lambda\dash\tn{Mod}^{\tn{abs}}}(M,N).$$
\end{lemma}

\begin{proof}
By taking the Pontryagin dual of a free resolution of the pseudocompact module $N^*$, we obtain an injective resolution of $N$ by modules that are sums of modules of the form $\Lambda^* = \tn{Hom}_{R\dash\tn{Mod}}^{\tn{cts}}(\Lambda, E_R)$.  Since $\Lambda$ is finitely generated over $R$, we have $\tn{Hom}_{R\dash\tn{Mod}}^{\tn{cts}}(\Lambda, E_R) = \tn{Hom}_{R\dash\tn{Mod}}^{\tn{abs}}(\Lambda, E_R)$.  The injective hull in $R\dash\tn{Mod}^{\tn{abs}}$ of the module $R/\mathfrak{m}$ ($\mathfrak{m}$ maximal ideal of $R$) is locally finite, since every element is annihilated by some power of $\mathfrak{m}$ by \cite[Thm 18.4]{Matsumura} and $R/\mathfrak{m}^n$ has finite length over $R$ for any $n\in \N$ because $R$ is noetherian.  Hence the injective hulls of $R/\mathfrak{m}$ in the abstract and the discrete categories coincide.  Furthermore, $R$ being noetherian implies that a direct sum of injective modules in the abstract category is injective \cite[Thm 18.5]{Matsumura}.  We conclude that $E_R$ is injective in $R\dash\tn{Mod}^{\tn{abs}}$.

Now, the functor $\tn{Hom}_{R\dash\tn{Mod}}^{\tn{abs}}(\Lambda,-):R\dash\tn{Mod}^{\tn{abs}}\to \Lambda\dash\tn{Mod}^{\tn{abs}}$ is a right adjoint by the theory of abstract modules, and hence by \cite[Prop.\ 2.3.10]{Weibel}, $\Lambda^*$ is injective in $\Lambda\dash\tn{Mod}^{\tn{abs}}$. 

Putting all this together, the above injective resolution of $N$ in $\Lambda\dash\tn{Mod}^{\mathcal{D}}$ is also an injective resolution of $N$ in $\Lambda\dash\tn{Mod}^{\tn{abs}}$.  Furthermore, the hom groups obtained by applying $\tn{Hom}_{\Lambda\dash\tn{Mod}^{\mathcal{D}}}(M,-)$ and $\tn{Hom}_{\Lambda\dash\tn{Mod}^{\tn{abs}}}(M,-)$ to this injective resolution are the same and so the Ext groups are the same. 
\end{proof}

As usual, when $G$ is a finite group we denote by $H^n(G,-)$ the $n$th right derived functor of the fixed point functor $(-)^G$.  We require only the following basic version of the Eckmann-Shapiro Lemma, which can be proved in the usual way given that the induction and coinduction functors coincide for finite groups: When $H$ is a subgroup of $G$ and $V$ is a pseudocompact $RH$-module, then $H^n(H,V) \iso H^n(G,V\ind{G})$ for all $n\geqslant 1$.

\subsection{Permutation modules}\label{subsection permutation modules}

Let $G$ be a profinite group and $R$ a commutative pseudocompact ring.  Let $(X,*)$ be a pointed profinite $G$-space (that is, $X$ is a profinite space on which $G$ acts continuously and such that $g* = *\; \forall g\in G$).  Recall (cf.\ {\cite[1.7]{me}}) that the free pseudocompact $R$-module on $(X,*)$ is a pseudocompact $R$-module $R\db{(X,*)}$ together with a continuous map $\iota : (X,*)\to R\db{(X,*)}$ sending $*$ to 0 and satisfying the following universal property:

Given any pseudocompact $R$-module $M$ and continuous map $\beta:X\to M$ sending $*$ to $0$, there is a unique continuous homomorphism of $R$-modules $\beta':R\db{(X,*)}\to M$ such that $\beta'\iota = \beta$.

\begin{defn}[cf.\ {\cite[Def.\ 1.8]{me}}]
Let $G$ be a profinite group and $R$ a commutative pseudocompact ring.  Let $(X,*)$ be a pointed profinite $G$-space.  The corresponding \emph{pseudocompact permutation module for $G$} is the $R$-module $R\db{(X,*)}$ with action from $G$ given by the action of $G$ on $X$ and the universal property of $R\db{(X,*)}$. 
\end{defn}

Note that the corresponding map $\iota$ is thus $G$-equivariant.  The definition can be succinctly stated as follows: a pseudocompact permutation $R\db{G}$-module is a pseudocompact $R\db{G}$-module having a pointed profinite $G$-invariant $R$-basis.  By construction, a permutation module is free as an $R$-module, so in particular is a lattice.  

\begin{lemma}\label{lemma UP perm module}
The permutation $R\db{G}$-module $R\db{(X,*)}$ is determined by the following universal property:

Given any pseudocompact $R\db{G}$-module $M$ and any continuous $G$-equivariant map $\beta:X\to M$ sending $*$ to $0$, there is a unique continuous $R\db{G}$-module homomorphism $\beta':R\db{(X,*)}\to M$ such that $\beta'\iota = \beta$.
\end{lemma}

\begin{proof}
That there exists a unique $R$-module homomorphism is the universal property of $R\db{(X,*)}$ as an $R$-module.  Using the universal property again, we see that this map is an $RG$-module homomorphism.
\end{proof}

We give several alternative definitions of a permutation module for a profinite group $G$, which in general are not equivalent.  Given a profinite group $G$ and an (unpointed) profinite $G$-space $X$, the corresponding discrete permutation module $F(X)$ is defined in \cite{symondspermcom1} to be the module of continuous functions $X\to T$, where $T$ is the maximal submodule in $R\dash\tn{Mod}^{\mathcal{D}}$ of the injective hull of $R$ as an abstract $R$-module.  The pseudocompact permutation module $R\db{X}$ in the sense of \cite{symondspermcom1} is the Pontryagin dual of $F(X)$.  This is the same as saying that $R\db{X}$ is the pseudocompact $R\db{G}$-module having profinite $R$-basis $X$.  Note that the modules $R[X]$ and $R[(X\cup\{*\},*)]$ are clearly isomorphic when $X$ is finite.  A profinite $G$-space $X$ (resp.\ profinite pointed $G$-space $(X,*)$) can be expressed as the inverse limit of an inverse system of finite (resp.\ finite pointed) $G$-sets $X_i$ (resp.\ $(X_i,*)$) \cite[Lem.\ 5.6.4]{RZ12010}.  We have that $R\db{X} \iso \invlim R[X_i]$ and $R\db{(X,*)} = \invlim R[(X_i,*)]$.  A \emph{strict} pseudocompact permutation module for $G$ is a module isomorphic to one of the form $\prod_{j\in J} R\db{G/H_j}$, where $J$ is a set and each $H_j$ is a closed subgroup of $G$.

\begin{prop}\label{Prop perm module characterization}
Let $R$ be a commutative local pseudocompact ring with residue class field of characteristic $p$.  Let $G$ be a finite $p$-group and let $U$ be a pseudocompact $RG$-module.  The following are equivalent:
\begin{enumerate}
\item\label{char U perm} $U$ is a permutation module,
\item\label{char U G-inv basis} $U$ has a profinite $G$-invariant $R$-basis,
\item\label{char U strict perm} $U$ is a strict pseudocompact permutation module,
\item\label{char U inv lim of perm} $U$ is an inverse limit of finite rank permutation modules.
\end{enumerate}
\end{prop}

\begin{proof}
Given a module of the form  $\prod_{j\in J} R[G/H_j]$, choose for each $j\in J$ the basis of left cosets of $H_j$ in $G$, and let $X$ be the union of these bases (a discrete set) compactified at the point $*$.  One now checks that $\prod_{j\in J} R[G/H_j] \iso R\db{(X,*)}$ by observing that the universal property of Lemma \ref{lemma UP perm module} follows from the universal property of the product.  Thus \ref{char U strict perm} implies \ref{char U perm}.  The equivalence of \ref{char U strict perm} and \ref{char U inv lim of perm} follows from \cite[Thm 2.2]{infgen} in case $R$ is profinite (note that \cite[Cor.\ 2.3]{infgen} is not correct as stated, since the module $R[G/H]$ need not be indecomposable when the residue class field of $R$ is not of characteristic $p$).  The same proof works for pseudocompact $R$ (see final remark of \S \ref{Section Change of category}).  As the $G$-space $X$ (resp.\ pointed $G$-space $(X,*)$) is the inverse limit of finite $G$-sets (resp.\ finite pointed $G$-sets), \ref{char U G-inv basis} (resp.\ \ref{char U perm}) implies \ref{char U inv lim of perm}.  We show that \ref{char U strict perm} implies \ref{char U G-inv basis}. As $G$ is finite, we may suppose that $U = \prod_J R[G/H]$ for some fixed subgroup $H$ of $G$ and discrete set $J$.  The result is obvious when $J$ is finite so suppose that $J$ is infinite.  
Let $J^*$ be the one-point compactification of $J$ with point at infinity $*$.  Consider the (unpointed) profinite $G$-space $G/H\times J^*$, where $G$ acts only on the left factor, in the obvious way.  We obtain 
\begin{align*}
R\db{G/H \times J^*} & \iso \invlim_{\substack{F\subseteq J \\ F\hbox{ finite}}}R[G/H\times(*\cup F)] \\
& \iso \invlim_{\substack{F\subseteq J \\ F\hbox{ finite}}} R[G/H] \times \prod_F R[G/H]   \\
& \iso R[G/H] \times \prod_J R[G/H] \\ 
& \iso U \quad (\hbox{since }J\hbox{ is infinite}).
\end{align*}
\end{proof}

Even for finite groups we can not, in general, suppose that the module $R[G/H]$ is indecomposable, and indeed the class of permutation modules need not be closed under summands.  We define an \emph{$R$-permutation module} to be a summand of a permutation module.  For a profinite group $G$, denote by $\tn{Perm}(G)$ the module $\prod_{H\leqslant G}R\db{G/H}$ (subgroups of profinite groups are always taken to be closed).  A strict pseudocompact $R$-permutation module for $G$ is an object of $\tn{Add}^{\mathcal{C}}(\tn{Perm}(G))$.  The second author showed \cite[Cor.\ 3.21]{symondspermcom1} that every $R$-permutation module for a finite group is strict.

\begin{lemma}\label{perm modules coflasque}
Let $G$ be a finite group and suppose that $|G|$ is not a zero divisor of $R$.  If $V$ is a pseudocompact $R$-permutation module then $H^1(H,V)=0$ for all $H\leqslant G$.
\end{lemma}

\begin{proof}
By the Eckmann-Shapiro Lemma, additivity and the Mackey decomposition formula, we need only check that $H^1(H,V)=0$ for $V$ a trivial pseudocompact $RH$-module.  But $H^1(H,V)$ is isomorphic to group homomorphisms from $H$ to $V$.  If $\rho$ is such a homomorphism then for any $h\in H$ we have
$$0 = \rho(1) = \rho(h^{|G|}) = |G|\rho(h)\implies \rho(h) = 0$$
since $|G|$ is not a zero-divisor.
\end{proof}

Permutation modules of the form $R\db{G/H}$ can be written as $R\ind{G}_H$, the trivial $RH$-module induced up to $G$.  A larger class of modules of interest to us is the class of \emph{monomial} modules (called generalized permutation modules in \cite{weiss}), namely the category of continuous direct summands of direct products of modules of the form $V\ind{G}_H$, where $V$ is an $R$-rank 1 $RH$-lattice.  When $G$ is finite and $R$ is an integral domain there are only finitely many isomorphism classes of lattices of $R$-rank 1 (because $R$ has only finitely many $|G|$th roots of unity), so the category of monomial modules is $\tn{Add}^{\mathcal{C}}(M)$ with $M = \bigoplus_{\substack{H\leqslant G}}V_H\ind{G}_H$, where $V_H$ runs through the set of isomorphism classes of $R$-rank 1 $RH$-lattices.  Notice that  monomial modules are preserved under induction and restriction.

\section{Change of category}\label{Section Change of category}

Recall that $\Lambda\dash\tn{Mod}^{\mathcal{C}}$ has exact inverse limits \cite[Ch.IV, Thm 3]{GabrielThesis} and $\Lambda\dash\tn{Mod}^{\mathcal{D}}$ has exact direct limits.  Let $M$ be a finitely generated object of $\Lambda\dash\tn{Mod}^{\mathcal{C}}$ and $E = \tn{End}_{\Lambda}(M)$, a pseudocompact algebra by Lemma \ref{Lemma E pseudocompact}.  Treating $M$ as a left $E$-module in the obvious way the actions of $\Lambda$ and $E$ are compatible, in so far as $\rho(\lambda m) = \lambda \rho(m)$ for $\rho\in E, \lambda\in \Lambda, m\in M$.  In what are perhaps more familiar terms, this amounts to saying that $M$ is a $\Lambda\dash E^{\tn{op}}$-bimodule.  Consider the following functors:
 \begin{align*}
 U = \tn{Hom}_{\Lambda}(M,-) & : \Lambda\dash\tn{Mod}^{\mathcal{C}} \to \tn{Mod}^{\mathcal{C}}\dash E \\
 V = -\ctens_E M & :\tn{Mod}^{\mathcal{C}}\dash E \to \Lambda\dash\tn{Mod}^{\mathcal{C}}.
 \end{align*}
 
 \begin{lemma}\label{Adjunction Lmod and Emod}
The functor $V$ is left adjoint to $U$.
\end{lemma}

\begin{proof}
See the discussion after Lemma \ref{Lemma E pseudocompact}.
\end{proof}

\begin{prop}\label{Prop restriction to Add and Proj Equiv}
Restricting the domain and codomain of $U$ and $V$ to $\tn{Add}^{\mathcal{C}}(M)$ and $\tn{Proj}^{\mathcal{C}}\dash E$ yields mutually inverse equivalences of categories.
\end{prop}

\begin{proof}
This is more or less standard and the usual proof goes through (see, for instance, \cite[Lem.\ 29.4]{AndersonFuller}).  The only difference is that we must check that the functors $U$ and $V$ commute with products rather than sums.  But this follows from observations in \S \ref{subsection tensor hom and Ext}.
\end{proof}

With $M$ still a finitely generated module in $\Lambda\dash\tn{Mod}^{\mathcal{C}}$, we may dualize the functors above to obtain
\begin{align*}
U' = \tn{Hom}_{\Lambda}(M,(-)^*)^* 
= \tn{Hom}_{\Lambda}(-,M^*)^* 
& : \tn{Mod}^{\mathcal{D}}\dash\Lambda \to E\dash\tn{Mod}^{\mathcal{D}} \\
V' = ((-)^*\ctens_E M)^* 
= \tn{Hom}_E(M,-)& : E\dash\tn{Mod}^{\mathcal{D}} \to \tn{Mod}^{\mathcal{D}}\dash\Lambda.
 \end{align*}

\begin{lemma}\label{Adjunction dual modL and modE}
The functor $V'$ is right adjoint to $U'$.
\end{lemma}

\begin{lemma}\label{Lemma LAPC and RAPL for Vs}
The functor $V$ commutes with direct limits and the functor $V'$ commutes with inverse limits.
\end{lemma}

\begin{proof}
These are formal properties of adjoint functors, see for instance \cite[Thm 2.6.10]{Weibel}.
\end{proof}

Again, by duality, restricting the domain and codomain, we obtain an equivalence of categories between $\tn{Add}^{\mathcal{D}}(M^*)$ and $E\dash\tn{Inj}^{\mathcal{D}}$.

These equivalences allow one to deduce properties of $\tn{Add}(M)$ from already understood properties of $E\dash\tn{Proj}^{\mathcal{C}}$ or $\tn{Inj}^{\mathcal{D}}\dash E$.  For example, from \cite[Ch.III, Cor.1 to Thm 3]{GabrielThesis} it follows that every module in $\tn{Add}^{\mathcal{C}}(M)$ is a product of indecomposable summands of $M$.  Furthermore, from  \cite[Ch.IV, Thm.2]{GabrielThesis} this decomposition is essentially unique.  Indeed, $\tn{Add}^{\mathcal{C}}(M)$ has the \emph{exchange property}: 

\begin{prop}[\tn{\cite[Ch. IV, Prop. 8]{GabrielThesis}}]\label{prop exchange property}
Let $X$ be a continuous direct summand of the pseudocompact $\Lambda$-module $Y = \prod_{i\in I} Y_i$, where each $Y_i$ is an indecomposable summand of the finitely generated pseudocompact $\Lambda$-module $M$.  There exists a subset $J$ of $I$ such that
$$Y = X \oplus \prod_{i\in J}Y_i.$$
\end{prop}

\begin{lemma}[\tn{\cite[Lem.\ A.4]{Brumer}}]\label{V commutes with limits}
The functor $V:\tn{Mod}^{\mathcal{C}}\dash E \to \Lambda\dash\tn{Mod}^{\mathcal{C}}$ commutes with inverse limits.
\end{lemma}

\begin{prop}[\tn{\cite[Cor.\ 3.3]{Brumer}, \cite[Ch.II, Cor.1 to Thm 2]{GabrielThesis}}]\label{inverse limit of projectives is projective}
An inverse limit of projective modules in $\Lambda\dash\tn{Mod}^{\mathcal{C}}$ is again a projective module.
\end{prop}

\begin{theorem}\label{thm limits of Add are in Add}
Let $\Lambda$ be a pseudocompact $R$-algebra and $M$ a finitely generated pseudocompact left $\Lambda$-module.
The inverse limit in $\Lambda\dash\tn{Mod}^{\mathcal{C}}$ of modules in $\tn{Add}^{\mathcal{C}}(M)$ is again an object of $\tn{Add}^{\mathcal{C}}(M)$.
\end{theorem}

\begin{proof}
The proof of \cite[Thm 2.2]{infgen} works in this greater generality, but this also follows easily from the results above: an inverse system in $\tn{Add}^{\mathcal{C}}(M)$ yields, by applying $U$, an inverse system of projective modules, whose limit is projective by Proposition \ref{inverse limit of projectives is projective}.  Applying $V$ to this limit we obtain an object of $\tn{Add}^{\mathcal{C}}(M)$.  But $V$ commutes with inverse limits by Lemma \ref{V commutes with limits}, and hence this module is the inverse limit of the original inverse system. 
\end{proof}

\noindent \emph{Remark:} This result is false without the topology.  In \cite{Bergman}, Bergman shows that any abstract module can be expressed as an inverse limit of injective modules.  This is most striking when $\Lambda$ is self-injective (for instance, when $R$ is a field and $\Lambda$ is the group algebra of a finite group) \cite[\S 1.6]{Benson}.  In this case projective modules are the same as injective modules, so that every abstract module is an inverse limit of projective modules.

\medskip

\noindent \emph{Remark:} Let $R$ be a finite unramified extension of the $p$-adic integers $\Z_p$ and let $G$ be a finite group.  When $G$ is cyclic of order $p$ or $p^2$ it is known that there are only finitely many isomorphism classes of finitely generated indecomposable $RG$-lattices \cite[Thm 33.7, Cor. 33.3a]{CR vol1}.  It follows that for general finite $G$ there are finitely many isomorphism classes of indecomposable finitely generated $RG$-lattices that are projective relative to subgroups of order $p$ or $p^2$ (for the definition of relative projectivity see after Lemma \ref{Lemma db Q killed for non conjugate perm module}, for further details see for example \cite[\S 3.6]{Benson}, or \cite{rejprojprofinite} for profinite modules).  Denote by $M_G$ the sum of (a representative of each isomorphism class of) the indecomposable $RG$-lattices projective relative to cyclic subgroups of order $p$ or $p^2$.  If $G$ is cyclic of order $p$ or $p^2$, the category $\tn{add}(M_G)$ coincides with the category of finitely generated $RG$-lattices.  As pseudocompact $RG$-lattices are inverse limits of finite rank $RG$-lattices by Lemma \ref{RG lattice limit of finite rank RG lattices}, Theorem \ref{thm limits of Add are in Add} tells us that compact $RG$-lattices are summands of products of indecomposables, hence products of indecomposables by the exchange property (Proposition \ref{prop exchange property}).
For general $G$, a compact $RG$-lattice $L$ that is projective relative to cyclic subgroups of order $p$ or $p^2$ is a direct summand of 
$$\prod_{\substack{C\leqslant G \\ C\iso C_p\hbox{ or }C_{p^2}}}L\res{C}\ind{G}$$ 
and hence (by the theorem) is an object of $\tn{Add}(M_G)$.

\medskip

\noindent \emph{Remark:} Every profinite $G$-set $X$ can be expressed as an inverse limit of finite $G$-sets $X = \invlim X_i$ \cite[Lem.\ 5.6.4(a)]{RZ12010} and hence $R\db{X} = \invlim R[X_i]$.  Theorem \ref{thm limits of Add are in Add} now yields a different proof of the result of \cite[Corol.\ 3.21]{symondspermcom1}: a pseudocompact permutation module can be expressed as a summand of a product of modules isomorphic to $R[G/H]$ ($H\leqslant G$).

\section{Other limits}

We have observed that $\Lambda\dash\tn{Mod}^\mathcal{C}$ has exact inverse limits and dually that $\Lambda\dash\tn{Mod}^\mathcal{D}$ has exact direct limits.  It is interesting to note that $\Lambda\dash\tn{Mod}^\mathcal{C}$ also has direct limits and $\Lambda\dash\tn{Mod}^\mathcal{D}$ has inverse limits (though of course, they are not exact) \cite[Ch.II, Cor.\ 2 to Thm 1]{GabrielThesis}.  Inverse limits in $\Lambda\dash\tn{Mod}^\mathcal{D}$ are obtained by taking the inverse limit in $\Lambda\dash\tn{Mod}^{\tn{abs}}$, giving it the discrete topology and then taking the submodule generated by elements annihilated by an open ideal of $\Lambda$.  Direct limits in $\Lambda\dash\tn{Mod}^\mathcal{C}$ are described in \cite[\S 2]{Corob Cook Bieri}.  Note that these limits do not in general commute with restriction to subgroups.

\medskip

By definition, $\Lambda = \invlim \Lambda/I_j$ for some set of open ideals $I_j$ of $\Lambda$ with each $\Lambda/I_j$ an $R$-algebra of finite length over $R$.  Given such an ideal $I$ and a module $M$ (compact, discrete torsion or abstract), denote by $T^IM$ the set $\{m\in M\,|\,Im=0\}$.  Note that $T^I\invlim^{\mathcal{D}} = \invlim^{\mathcal{D}}T^I$.  Let $TM$ denote the union of the $T^IM$ as $I$ varies over the open ideals of $\Lambda$.

\begin{prop}\label{prop wrong way limits of projectives and injectives}
A direct limit in $\Lambda\dash\tn{Mod}^\mathcal{C}$ of projective modules in $\Lambda\dash\tn{mod}^\mathcal{C}$ is again projective.  An inverse limit in $\Lambda\dash\tn{Mod}^\mathcal{D}$ of injective modules in $\Lambda\dash\tn{mod}^\mathcal{D}$ is again injective.
\end{prop}

\begin{proof}
We prove the second statement.  First suppose that $\Lambda$ has finite length over $R$.  An inverse system in $\tn{add}^{\mathcal{D}}(\Lambda^*)$ can be treated as an inverse system in $\tn{add}^{\mathcal{C}}(\Lambda^*)$, and so by Theorem \ref{thm limits of Add are in Add} it has an inverse limit in $\tn{Add}^{\mathcal{C}}(\Lambda^*)$.  The module obtained by giving this module the discrete topology coincides with the inverse limit in $\Lambda\dash\tn{Mod}^\mathcal{D}$ because all these modules are torsion.

In $\tn{Add}^{\mathcal{C}}(\Lambda^*)$ our module has the form $\bigoplus I_j\ctens_R V_j$, where $I_j$ runs through the indecomposable summands of $\Lambda^*$ and $V_j$ is a free module in $R\dash\tn{Mod}^{\mathcal{C}}$.  But each $I_j$ is of finite length, hence finitely presented, so that $I_j\ctens V_j\iso I_j\otimes V_j$ by Proposition \ref{prop tensors coincide}.  If we ignore the topology, $V_j$ is a free abstract module, by \cite[Thm 3.3]{Chase}.  Thus, the module is in $\tn{Add}^{\mathcal{D}}(\Lambda^*) = \tn{Inj}^{\mathcal{D}}\dash \Lambda$.

\medskip

For the general case, write $\Lambda = \invlim \Lambda/I_i$.  Denote by $X = \invlim X_j$ our module, where each $X_j$ is injective.  Then
$$T^{I_i}\invlim^{\mathcal{D}}X_j = 
\invlim^{\mathcal{D}}T^{I_i}X_j.$$
But each $T^{I_i}X_j$ is injective as a $\Lambda/I_i$-module, so by the special case above $\invlim_j^{\mathcal{D}}T^{I_i}X_j$ is in $\tn{Inj}^{\mathcal{D}}\dash\Lambda/I_i$.  Now take $\dirlim_{I_i}$, observing that by \cite[Ch.2, Cor.\ 1 to Thm 1]{GabrielThesis} a direct limit of modules injective over $\Lambda/I_i$ is itself injective over $\Lambda$.
\end{proof}

\begin{theorem}
A direct limit in $\Lambda\dash\tn{Mod}^\mathcal{C}$ of modules in $\tn{add}^{\mathcal{C}}(M)$ is in $\tn{Add}^{\mathcal{C}}(M)$.  An inverse limit in $\Lambda\dash\tn{Mod}^\mathcal{D}$ of modules in $\tn{add}^{\mathcal{D}}(M^*)$ is in $\tn{Add}^{\mathcal{D}}(M^*)$.
\end{theorem}

\begin{proof}
We prove the first statement.  Let $X = \dirlim X_i$ be a direct limit of modules in $\tn{add}^{\mathcal{C}}(M)$.  Then by Proposition \ref{Prop restriction to Add and Proj Equiv} there are projective right $E$-modules $P_i$ for which
$$X = \dirlim X_i = \dirlim V(P_i) = V(\dirlim P_i),$$
where the last equality is from Lemma \ref{Lemma LAPC and RAPL for Vs}.  Hence $X\in \tn{Add}^{\mathcal{C}}(M)$ by Propositions \ref{prop wrong way limits of projectives and injectives} and \ref{Prop restriction to Add and Proj Equiv}.
\end{proof}

\section{Flatness}

We say that a module $X$ in $\Lambda\dash\tn{Mod}^{\mathcal{C}}$ is \emph{flat} if the functor $-\ctens_{\Lambda}X$ is exact on $\tn{Mod}^\mathcal{C}\dash\Lambda$.  A module $X$ in $\Lambda\dash\tn{Mod}^{\mathcal{C}}$ is flat if, and only if, it is projective.  This follows from the existence of projective covers, and so the usual proof for perfect rings (see eg. \cite[Thm 24.25]{Lam}) carries through.

We consider other criteria for flatness.  Observe that a direct limit of projective modules need not be projective, but by Proposition \ref{prop wrong way limits of projectives and injectives} a direct limit in $\Lambda\dash\tn{Mod}^{\mathcal{C}}$ of finitely generated projectives is again projective.  The converse is false: a projective module need not be a direct limit of finitely generated projectives.  This is easily seen in the dual case, by taking $\Lambda = R = k$, a finite field.  The sum $\bigoplus_{n\in \N}k$ of countably many copies of $k$ (an injective object of $\tn{Mod}^\mathcal{D}\dash k$) is countable, hence not an inverse limit of finite dimensional vector spaces by \cite[Prop.\ 2.3.1(b)]{RZ12010}.  

On the other hand, if we allow summands, the problem vanishes: If $X$ is an object of  $\tn{Inj}^{\mathcal{D}}\dash\Lambda$ then $X = \bigoplus_{j\in J}I_j$ with each $I_j$ an indecomposable injective.  Consider $Y = \prod_{j\in J}I_j$ in the category of \emph{abstract} $\Lambda$-modules.  Then $X\subseteq TY\subseteq Y$.  Because $Y$ is an inverse limit of finitely cogenerated injectives, it follows that $TY$ is an inverse limit in $\tn{Ind}^\mathcal{D}\dash\Lambda$.  The inclusion $X\to TY$ splits because $X$ is injective.

\section{Add$(M)$ precovers and covers}\label{Section precovers and covers}

It is convenient to follow the approach to covers and precovers given in \cite{EnochsJenda}.  In this section, $M$ denotes a finitely generated module in $\Lambda\dash\tn{Mod}^{\mathcal{C}}$.

\begin{defn}[cf. \hbox{\cite[\S 5]{EnochsJenda}}]
Let $X$ be an object of $\Lambda\dash\tn{Mod}^{\mathcal{C}}$.  The continuous homomorphism $\rho:P\to X$ is an \emph{$\tn{Add}(M)$-precover} of $X$ if
\begin{itemize}
\item $P\in \tn{Add}(M)$,
\item Given any continuous homomorphism $\alpha:S\to X$ with $S$ in $\tn{Add}(M)$, there exists a continuous homomorphism $\gamma:S\to P$ such that $\rho\gamma = \alpha$. 
\end{itemize}
An $\tn{Add}(M)$-precover $\rho:P\to X$ is an \emph{$\tn{Add}(M)$-cover} of $X$ if every continuous homomorphism $\gamma:P\to P$ such that $\rho\gamma=\rho$ is an automorphism of $P$. 
\end{defn}

The $\tn{Add}(M)$-cover of $X$, if it exists, is clearly unique up to isomorphism.  It is also easily checked that the $\tn{Add}(M)$-cover of $X$ is a direct summand of any $\tn{Add}(M)$-precover of $X$, and that an $\tn{Add}(M)$-precover is a split-surjection if, and only if, $X\in \tn{Add}(M)$.  It follows that an $\tn{Add}(\Lambda)$-cover corresponds to the usual notion of a projective cover -- that is, an $\tn{Add}(\Lambda)$-cover $\rho:P\to X$ of $X$ is precisely a surjective homomorphism from the projective $\Lambda$-module $P$ with $P$ minimal with respect to direct sum decompositions (cf.\ \cite[\S 1.5]{Benson}).

\begin{theorem}\label{covers exist}
Let $M$ be a finitely generated $\Lambda$-module and let $X$ be any pseudocompact $\Lambda$-module.  The $\tn{Add}(M)$-cover of $X$ exists.
\end{theorem}

\begin{proof}
Denote by $\varepsilon, \eta$ the counit and unit of the adjunction of Lemma \ref{Adjunction Lmod and Emod}.  By results mentioned there, restricting the domain and codomain of $\varepsilon$ to $\tn{Add}(M)$, $\tn{Proj}(E)$ yields a natural isomorphism.  Since $E$ is pseudocompact by Lemma \ref{Lemma E pseudocompact}, $\tn{Mod}^{\mathcal{C}}\dash E$ has projective covers. 

Let $\rho:P\onto U(X)$ be the projective cover of $U(X)$.  Then $V(P)$ is an $\tn{Add}(M)$-module and
$$\varepsilon_X\circ V(\rho) : V(P) \to VU(X) \to X$$
is the $\tn{Add}(M)$-cover of $X$.  We sketch the argument.  Given $\alpha:S\to X$ with $S\in \tn{Add}(M)$, there is a homomorphism $\gamma:U(S)\to P$ such that $\rho\gamma = U(\alpha)$ because $\rho$ is a precover.  The diagram
$$
\xymatrix{    
               &  V(P)\ar[d]^{\varepsilon_X V(\rho)} \\ 
S\ar[r]_{\alpha} \ar[ur]^{V(\gamma)\varepsilon_S^{-1}}  & X  } 
$$
commutes:
\begin{align*}
\varepsilon_X V(\rho)V(\gamma){\varepsilon_S}^{-1} & = \varepsilon_X V(\rho\gamma){\varepsilon_S}^{-1} \\
& = \varepsilon_X VU(\alpha){\varepsilon_S}^{-1} \\
& = \alpha \varepsilon_S{\varepsilon_S}^{-1} \\
& = \alpha,
\end{align*}
hence $\varepsilon_XV(\rho)$ is a precover.  To check it is a cover, consider $\gamma$ completing the following diagram:
 $$
 \xymatrix{    
V(P)\ar[rr]^{\gamma}\ar[dr]_{\varepsilon_XV(\rho)} &                &  V(P)\ar[dl]^{\varepsilon_XV(\rho)} \\ 
 &  X &}
 $$
One checks that ${\eta_P}^{-1}U(\gamma)\eta_P:P\to P$ is an isomorphism using the counit-unit equations and the fact that $\rho$ is a cover.  Hence $\gamma$ too is an isomorphism. 
\end{proof}

From the construction it follows that if $X$ is finitely generated over $\Lambda$, then so is its $\tn{Add}(M)$-cover.  In the following definition, $H\leqslant_G G$ indicates that $H$ runs through a set of representatives of the classes of $G$-conjugates of subgroups of $G$.

\begin{defn}
Let $G$ be a finite group and $X$ a pseudocompact $RG$-module.  A \emph{permutation precover} of $X$ is an $\tn{Add}(M)$-precover of $X$, where 
$$M = \prod_{H\leqslant_G G}R[G/H].$$ 
The \emph{permutation cover} of $X$ is the corresponding $\tn{Add}(M)$-cover.

Let $R$ be pseudocompact integral domain.  A \emph{monomial precover} (or \emph{generalized permutation precover}) of $X$ is an $\tn{Add}(M)$-precover of $X$, where
$$M  = \prod_{\substack{H\leqslant_G G}}V_H\ind{G}_H$$
and $V_H$ runs through the set of isomorphism classes of $R$-rank 1 $RH$-lattices.  The \emph{monomial cover} of $X$ is the corresponding $\tn{Add}(M)$-cover.
\end{defn}

Both the permutation and monomial covers of $X$ exist by Theorem \ref{covers exist}.

\medskip

We present some useful properties of permutation and monomial precovers that we will require in Section \ref{Section Weiss}.  Following Samy Modeliar \cite{SamyModeliarThesis}, we say that an $RG$-module homomorphism $\rho:V\to U$ is \emph{supersurjective} if the induced homomorphism $\rho^H:V^H\to U^H$ is surjective for every subgroup $H$ of $G$.  Similarly, $\rho$ is \emph{monomial supersurjective} if the induced homomorphism $\tn{Hom}_{RG}(W\ind{G}_H,V)\to \tn{Hom}_{RG}(W\ind{G}_H,U)$ is surjective for every $R$-rank $1$ $RH$-lattice $W$.  Note that monomial supersurjective homomorphisms are supersurjective.

\begin{lemma}\label{perm precovers supersurjective}
\begin{enumerate}
\item A homomorphism from a permutation module is a permutation precover if, and only if, it is supersurjective.
\item Let $\pi:X\onto Y$ be supersurjective and $f:L\to Y$ a homomorphism from a permutation module $L$.  Then $f$ lifts to a homomorphism $\widetilde{f}:L\to X$ such that $\pi\widetilde{f}=f$.
\end{enumerate}
\end{lemma}

\begin{proof}
\begin{enumerate}
\item Given a subgroup $H$ of $G$, we have isomorphisms of functors
$$\tn{Hom}_{RG}(R\ind{G}_H,-)\iso \tn{Hom}_{RH}(R,(-)\res{H}) \iso (-)^H.$$
Suppose first that $\rho : P \to U$ is a permutation precover. Via the above isomorphisms, an element $u\in U^H$ is $\alpha(H)$ for a homomorphism $\alpha : R[G/H] \to U$. This homomorphism factors as $\alpha = \rho\gamma$ for a homomorphism $\gamma:R[G/H]\to P$ and hence $\gamma(H)$ is an element of $P^H$ mapping onto $u$ via $\rho^H$, showing that $\rho^H$ is surjective.

Suppose now that $\rho$ is supersurjective.  An $R$-permutation module is a finite sum of isotypic components so it suffices to check that we can lift a homomorphism $\alpha : X \to U$ where $X$ is a direct summand of $\prod R[G/H]$ for some subgroup $H$ of $G$.  Let $\iota : X \to \prod R[G/H]$, $\pi : \prod R[G/H] \to X$ be the splitting maps. By the functor isomorphisms above and supersurjectivity, there is a homomorphism $\gamma : \prod R[G/H] \to P$ such that $\rho\gamma = \alpha\pi$, and now $\gamma\iota$ is the required lifting of $\alpha$.

\item Let $\alpha:P\to X$ be a permutation precover of $X$.  Then $\pi\alpha$ is supersurjective, hence a permutation precover of $Y$.  It follows that there exists a $\gamma:L\to P$ such that $\pi\alpha\gamma=f$, and now $\alpha\gamma = \widetilde{f}$ is the required lift.
\end{enumerate}
\end{proof}

In just the same way we have:

\begin{lemma}\label{monomial precovers monomial supersurjective}
Let $R$ be a pseudocompact integral domain.
\begin{enumerate}
\item A homomorphism from a monomial module is a monomial precover if, and only if, it is monomial supersurjective.

\item Let $\pi:X\onto Y$ be monomial supersurjective and $f:L\to Y$ a homomorphism from a monomial module $L$.  Then $f$ lifts to a homomorphism $\widetilde{f}:L\to X$ such that $\pi\widetilde{f}=f$.
\end{enumerate}
\end{lemma}

\begin{lemma}\label{Lemma supersurj onto Rperm is split}
If $M$ is an $R$-permutation module and $\rho:L\to M$ is supersurjective, then $\rho$ splits.
\end{lemma}

\begin{proof}
Let $\gamma:C\to L$ be a permutation precover of $L$.  The composition $\rho\gamma$ is supersurjective and hence a permutation precover of the $R$-permutation module $M$.  Thus $\rho\gamma$ splits, hence so does $\rho$.
\end{proof}

A lattice $L$ in $RG\dash\tn{Mod}^{\mathcal{C}}$ is said to be \emph{coflasque} if $H^1(H,L)=0$ for all $H\leqslant G$.

\begin{prop}
Let $0\to L\to M\xrightarrow{\rho} N\to 0$ be a short exact sequence of $RG$-lattices.  Suppose that $N$ is an $R$-permutation module and that $L$ is coflasque.  Then the sequence splits.
\end{prop}

\begin{proof}
For any subgroup $H$ of $G$ the sequence $0\to L^H\to M^H\xrightarrow{\rho^H} N^H\to 0$ is exact because $H^1(H,L)=0$.  Thus $\rho$ is supersurjective and hence split by Lemma \ref{Lemma supersurj onto Rperm is split}.
\end{proof}

\begin{corol}
Let $0\to L\to M\to N\to 0$ be a short exact sequence of $RG$-modules with $L,N$ both $R$-permutation modules.  If $|G|$ is not a zero divisor of $R$, then the sequence splits.
\end{corol}

\begin{proof}
By Lemma \ref{perm modules coflasque}, $L$ is coflasque, so the result is immediate from the above proposition.
\end{proof}

We will give a formula for the permutation cover of a pseudocompact $RG$-module using methods similar to those developed by Samy Modeliar \cite{SamyModeliar}, when $R$ is a complete discrete valuation ring whose residue class field has characteristic $p$ and $G$ is a finite group.  Let $\pi$ be a prime element of $R$.  Given a pseudocompact $RG$-module $X$ and a subgroup $H$ of $G$, consider the map
\begin{align*}
\tn{Tr}_H^G : X^H & \to X^G \\
x\,\,\,\, & \mapsto \sum_{g\in G/H}gx 
\end{align*}
where $g$ runs through a set of left coset representatives of $G/H$.  This is a special case of the map defined in \cite[Def.\ 3.6.2]{Benson} and the properties given there apply.  Using \cite[Lem.\ 3.6.3]{Benson} one sees that if $|G/H|$  is coprime to $p$, then $\tn{Tr}_H^G$ is surjective. 

\begin{defn}
Given a $p$-subgroup $P$ of $G$ and a pseudocompact $RG$-module $X$, define
$$X^{[P]} :=  \frac{X^P}{\sum_{Q<P}\tn{Tr}_Q^P (X^Q) + \pi {X}^P}$$
$$X^{\db{P}} :=  \frac{X^{P}}{\sum_{\substack{P<Q \\ Q\, p-\hbox{group}}}X^Q + \sum_{Q<P}\tn{Tr}_Q^P (X^Q) + \pi {X}^P}$$
wherein the symbol $<$ indicates strict inclusion.
\end{defn}
 
Note that in $X^{[P]}$ and $X^{\db{P}}$ we quotient out by $\pi X^P$, so that both are $(R/\pi)[\norm{G}(P)/P]$-modules.  The module $X^{[P]}$ is studied in \cite{Broue} and references therein, while $X^{\db{P}}$ is inspired by \cite[Def.\ 5]{SamyModeliar}.  An $RG$-module homomorphism $f:X\to Y$ sends $X^P$ into $Y^P$ and the denominator of $X^{[P]}$ into the denominator of $Y^{[P]}$, and hence induces a natural map $f^{[P]} : X^{[P]}\to Y^{[P]}$.  Similarly, $f$ induces a natural map $f^{\db{P}} : X^{\db{P}}\to Y^{\db{P}}$.  

\medskip

\noindent \emph{Remark:} Note that $X^{[P]}\res{P/P} = (X\res{P})^{[P]}$.  Such a strong statement cannot be made for $X^{\db{P}}$ but we do at least have that $X^{\db{P}} = (X\res{\norm{G}(P)})^{\db{P}}$ because if $Q$ is a $p$-subgroup of $G$ properly containing $P$, then so is $Q\cap \norm{G}(P) = \norm{Q}(P)$.

\begin{lemma}\label{Lemma check supersurjective}
The homomorphism $f:M\to N$ in $RG\dash\tn{Mod}^{\mathcal{C}}$ is supersurjective if, and only if, $f^{\db{P}}:M^{\db{P}}\to N^{\db{P}}$ is surjective for every $p$-subgroup $P$ of $G$. 
\end{lemma}

\begin{proof}
The forward implication is clear.  If $f^H$ is not surjective for some subgroup $H$ of $G$, one easily checks using the surjectivity of $\tn{Tr}_P^H$ that $f^{P}$ is also not surjective for $P$ any Sylow $p$-subgroup of $H$.  Thus $f$ is supersurjective if, and only if, $f^P$ is surjective for every $p$-subgroup of $G$.  Note further that if $f^P$ is surjective for every subgroup $P$ contained in a fixed Sylow $p$-subgroup $S$ of $G$, then $f$ is supersurjective (since $M^{{}^g\!P} = gM^P$).  Thus we restrict once and for all to a Sylow $p$-subgroup $S$ of $G$.

Given a $p$-subgroup $P$ of $S$, we will show that
$$f^{\db{P}}\hbox{ surjective }\forall P\implies f^{[P]}\hbox{ surjective }\forall P\implies f^{P}\hbox{ surjective }\forall P.$$
We prove the second implication by induction on the order of $P$.  The base case $P=1$ is Nakayama's Lemma.  Fix a $p$-subgroup $P$ of $S$ and suppose by induction that $f^Q$ is surjective for every proper subgroup $Q$ of $P$.  As $f^{[P]}$ is surjective, we can write a given $n\in N^P$ as $f(m) + \sum_{Q<P}\tn{Tr}_Q^P(n_Q) + \pi x$ for some $m\in M^P, n_Q\in N^Q$ and $x\in N^P$.  But $n_Q = f(m_Q)$ for some $m_Q\in M^Q$ by induction, and hence
$$n-\pi x = f(m) + \sum_{Q<P}\tn{Tr}_Q^P(f(m_Q)) = f(m + \sum_{Q<P}\tn{Tr}_Q^P(m_Q))\in f(M^P).$$
This shows that $f^P$ is surjective modulo $\pi$, so that $f^P$ is surjective by Nakayama's Lemma.  We prove the first implication by induction on the index of $P$ in $S$, the case $P=S$ being trivial. 
Consider the cokernel $C$ of $f^{[P]}$.  Then $C^{\db{P}}$ is the cokernel of $f^{\db{P}}$, so it is 0 by hypothesis.  We will show that $C_{\tn{N}_S(P)}=0$, which implies that $C=0$ since taking coinvariants by a $p$-group cannot kill a non-zero module.  We must check that
$$\sum_{\substack{P<Q \\ Q\, p-\hbox{group}}}C^Q\leqslant \sum_{\substack{Q<P}}\tn{Tr}_Q^P(C^Q)$$
in $C_{\tn{N}_S(P)}$. Consider $c\in C^Q$.  We may suppose without loss of generality that $|Q:P| = p$.  By induction, $c = \sum \tn{Tr}_T^Q(d_T)$ with $d_T\in C^T$ ($T$ a proper subgroup of $Q$).  Fix such a $d=d_T$ and suppose without loss of generality that $|Q:T|=p$.  Whenever $A\leqslant B$ are subgroups, denote by $\tn{Res}_A^B:C^B\to C^A$ the inclusion.  The Mackey formula yields
$$\tn{Res}_P^Q\tn{Tr}_T^Q(d) = \sum_{q\in P\backslash Q/T}\tn{Tr}_{P\cap{}^q\!T}^P\,q\tn{Res}_{P^q\cap T}^Td.$$
If $P\neq T$ then, since both are normal in $Q$, no conjugate of $T$ is equal to $P$ and hence the corresponding element is induced from proper subgroups of $P$ as required.  We are left with the case where $P=T$. The sum simplifies to
$$\tn{Res}_P^Q\tn{Tr}_T^Q(d) = \sum_{q\in Q/P}qd.$$
But $Q\subseteq \tn{N}_S(P)$ and hence $\sum_{q\in Q/P}qd$ is equal to $pd$ in $C_{N_S(P)}$, so is $0$.
\end{proof}

\begin{lemma}\label{Lemma db Q killed for non conjugate perm module}
Let $G$ be a finite group and $P,Q$ $p$-subgroups of $G$.  If $P$ is not $G$-conjugate to $Q$, then
$$R[G/P]^{\db{Q}} =0.$$
\end{lemma}

\begin{proof}
By the remark before Lemma \ref{Lemma check supersurjective} and the Mackey decomposition, we may restrict to $\norm{G}(Q)$ and hence suppose that $Q$ is normal in $G$.  There are two cases:
\begin{itemize}
    \item $P\cap Q < Q$.  The module $R[G/P]\res{Q}$ is a direct sum of modules of the form $R[Q/{}^{g}\!P\cap Q]$.  But ${}^{g}\!P\cap Q < Q$ and hence $R[G/P]^{[Q]}=0$ by the remark, so that $R[G/P]^{\db{Q}} = 0$ also.
    
    \item $P\cap Q = Q$, so that $Q<P$.  For any $g\in G$ the basis element $gP$ is in $R[G/P]^{{}^g\!P}$, hence $gP = 0$ in $R[G/P]^{\db{Q}}$ because $Q < {}^g\!P$.
\end{itemize}
\end{proof}

Here and in Section \ref{Section Weiss} we make use of the concept of relative projectivity.  Recall that an $RG$-module $U$ is projective relative to the subgroup $H$ if any continuous epimorphism of $RG$-modules $V\to U$ that splits as an $RH$-module homomorphism, splits as an $RG$-module homomorphism.  This is equivalent to saying that $U$ is a direct summand of an $RH$-module induced up to $G$, or to saying that the identity map of $U$ can be written as $\tn{Tr}_H^G(\alpha):=\sum_{x\in G/H}x\alpha x^{-1}$ for some continuous $RH$-endomorphism $\alpha$ of $U$ (the proof of \cite[Prop.\ 3.6.4]{Benson} goes through for pseudocompact modules). Dually, a module $U$ is injective relative to the subgroup $H$ if any continuous monomorphism of $RG$-modules $U\to V$ that splits as an $RH$-module homomorphism, splits as an $RG$-module homomorphism.  

If $U$ is an indecomposable finitely generated $RG$-module, a subgroup $Q$ of $G$ minimal with respect to the condition that $U$ is relatively $Q$-projective is called a \emph{vertex} of $U$.  The vertices of $U$ are $p$-subgroups and are conjugate in $G$ \cite[Prop.\ 3.10.2]{Benson}.  Furthermore, there is an indecomposable finitely generated $RQ$-module $S$ (unique up to conjugation in $\norm{G}(Q)$) such that $U\ds S\ind{G}$, called a \emph{source} of $U$ (the notation $X\ds Y$ indicates that the module $X$ is isomorphic to a direct summand of $Y$).  If $U$ has source the trivial module $R$, we call $U$ a \emph{trivial source module}.  Given a finitely generated indecomposable $R[\norm{G}(P)]$-lattice $E$ with vertex $P$, denote by $M(P,E)$ its Green correspondent \cite[Sec.\ 3.12]{Benson}.  Thus $M(P,E)$ is an indecomposable $RG$-lattice with vertex $P$ such that $M(P,E)\ds E\ind{G}$ and $E\ds M(P,E)\res{\norm{G}(P)}$.  Furthermore, $M(P,E)$ is the only summand of $E\ind{G}$ with vertex $P$ and $E$ is the only summand of $M(P,E)\res{\norm{G}(P)}$ with vertex $P$, the other summands of $E\ind{G}$ being projective relative to strictly smaller subgroups and the other summands of $M(P,E)\res{\norm{G}(P)}$ being projective relative to proper conjugates of $P$ \cite[Thm 3.12.2]{Benson}.  More generally, when $E$ is a product of finitely generated indecomposable $R[\norm{G}(P)]$-modules with vertex $P$, denote by $M(P,E)$ the product of the corresponding Green correspondents.  Note that $E$ is still a (continuous) summand of $M(P,E)\res{\norm{G}(P)}$ and $M(P,E)$ is still a (continuous) summand of $E\ind{G}$ -- this can be seen by direct calculation with the product topology.  

Let $E$ be a product of indecomposable trivial source $R[\norm{G}(P)]$-modules with vertex $P$ and let the $RG$-module $M(P,E)$ be as above.  Let $j : M(P,E)\to E\ind{G}$ be a split inclusion and define further the split inclusion $s : E\to E\ind{G}\res{\norm{G}(P)}$ by $s(e) = 1\ctens e$.

\begin{lemma}\label{Lemma map yields map from GC}
\begin{enumerate}
    \item The maps $j^{[P]}, s^{[P]}$ are isomorphisms.
    \item If $M$ is an $RG$-module and $f : E \to M\res{\norm{G}(P)}$ is an $R[\norm{G}(P)]$-module homomorphism, there exists an $RG$-module homomorphism $\widehat{f} : M(P,E)\to M$ such that
$$\widehat{f}^{[P]} = f^{[P]}\gamma,$$
where $\gamma$ is an isomorphism from $M(P,E)^{[P]}$ to $E^{[P]} = E/\pi E$.
\end{enumerate}
\end{lemma}

\begin{proof}
\begin{enumerate}
    \item We show the result for $s$, the argument for $j$ being similar.  Let $E\ind{G}\res{\norm{G}(P)} = s(E) \oplus Y$.  To prove the claim, it is sufficient to show that $Y^{[P]} = 0$.   By the Green correspondence, an indecomposable summand of $Y$ is a direct summand of a module of the form $R[\norm{G}(P)/H]$ for $H$ a subgroup of some proper conjugate ${}^{g}\!P$ of $P$.  We have
$$R[\norm{G}(P)/H]^{[P]}\res{P} = \left(R\ind{\norm{G}(P)}_H\res{P}\right)^{[P]}
\iso \bigoplus_{x\in P\backslash \norm{G}(P)/ H}\left(R\res{{}^x\!H\cap P}\ind{P}\right)^{[P]} = 0,$$
because each ${}^x\!H\cap P \subseteq {}^g\!P\cap P$ is a proper subgroup of $P$.

\item The map $f$ extends by Frobenius reciprocity to a unique $RG$-module homomorphism $\widetilde{f}: E\ind{G}\to M$ given on pure tensors by $\widetilde{f}(g\ctens e) = gf(e)$.  Note that $\widetilde{f}s=f$.  Define $\widehat{f} = \widetilde{f}j$.  We have
$$\widehat{f}^{[P]} = \widetilde{f}^{[P]}j^{[P]} = \widetilde{f}^{[P]}s^{[P]}((s^{[P]})^{-1}j^{[P]}) = f^{[P]}\gamma,$$
where $\gamma = (s^{[P]})^{-1}j^{[P]}$.
\end{enumerate}
\end{proof}

Let $X$ be a pseudocompact $RG$-module.  Given a $p$-subgroup $P$ of $G$, let $\varphi_P : E_P\to X^{\db{P}}$ be the projective cover of $X^{\db{P}}$ as an $R[\norm{G}(P)/P]$-module. By projectivity of $E_P$, $\varphi_P$ lifts to a homomorphism $E_P\to X^P$.  Composing with the inclusion we obtain a homomorphism $\theta_P : E_P\to X\res{\norm{G}(P)}$, which by Lemma \ref{Lemma map yields map from GC} yields an $RG$-module homomorphism $\widehat{\theta_P} : M(P,E_P)\to X$ such that $\widehat{\theta_P}^{[P]} = \theta_P^{[P]}\gamma_P$ for some isomorphism $\gamma_P : M(P,E_P)^{[P]}\to E_P^{[P]}$.  Note that $\varphi_P^{\db{P}}\gamma_P^{\db{P}} = \theta_P^{\db{P}}\gamma_P^{\db{P}} =  \widehat{\theta_P}^{\db{P}}$.
Denote by 
$$\theta : C = \prod_{P}M(P,E_P) \to X$$
the homomorphism that acts as $\widehat{\theta_P}$ on the factor $M(P,E_P)$, as $P$ ranges through a set of representatives of the $G$-conjugacy classes of $p$-subgroups of $G$.  For each $P$, denote by $m_P: M(P,E_P)\to C$ the inclusion, noting that $m_P^{\db{P}}$ is an isomorphism by Lemma \ref{Lemma db Q killed for non conjugate perm module}.

\begin{theorem}\label{Theorem explicit permutation cover}
The map $\theta$ is a permutation cover of $X$.
\end{theorem}

\begin{proof}
By Lemmas \ref{perm precovers supersurjective} and \ref{Lemma check supersurjective}, to check that $\theta$ is a precover we must check that $\theta^{\db{P}} : C^{\db{P}}\to X^{\db{P}}$ is surjective for any $p$-subgroup $P$ of $G$.  But
$$\theta^{\db{P}}m_P^{\db{P}} = \widehat{\theta_P}^{\db{P}} =  \varphi_P^{\db{P}}\gamma_P^{\db{P}},$$
and hence $\theta^{\db{P}}$ is surjective. 

It remains to check that $\theta$ is a cover.  If not, there exists a permutation precover $d : D\to X$ and a non-surjective homomorphism $f : D\to C$ such that $\theta f = d$.  By Lemma \ref{Lemma check supersurjective} there is a $p$-subgroup $P$ such that $f^{\db{P}}: D^{\db{P}}\to C^{\db{P}}$ is not surjective.  For this $P$, let $\rho: S\to D^{\db{P}}$ be a projective cover of $D^{\db{P}}$ as an $R[\norm{G}(P)/P]$-module.  By the projectivity of $S$ we obtain a map $\delta$ yielding the following commutative diagram:
$$
 \xymatrix{    
 &&&& S \ar[d]^{\rho}\ar[lllldd]_{\delta} & \\
 &&&& D^{\db{P}}\ar[d]^{f^{\db{P}}}\ar[drr]^{d^{\db{P}}} & \\
E_P \ar[r]^{\alpha}\ar@/_1.5pc/[rrrrrr]_{\varphi_P} & E_P^{\db{P}} \ar[rr]^-{(\gamma_P^{-1})^{\db{P}}} && M(P,E_P)^{\db{P}} \ar[r]^-{m_P^{\db{P}}} & C^{\db{P}}\ar[rr]^{\theta^{\db{P}}} && X^{\db{P}}
}
$$
The map $\alpha$ is the canonical projection from $E_P = {E_P}^P$ to $E_P^{\db{P}}$.  The lower shape commutes because $\varphi_P = \varphi_P^{\db{P}}\alpha$ and $\varphi_P^{\db{P}}\gamma_P^{\db{P}} = \theta^{\db{P}}m_P^{\db{P}}$.  Note that $\delta$ is surjective because $\varphi_P\delta = d^{\db{P}}\rho$ is a surjective map from the projective $R[\norm{G}(P)/P]$-module $S$ to $X^{\db{P}}$ and $\varphi_P$ is a projective cover of $X^{\db{P}}$.  This yields a contradiction, because the path from $S$ to $C^{\db{P}}$ going via $\delta$ is surjective, being a composition of surjective maps, while $f^{\db{P}}\rho$ is not surjective.
\end{proof}

The proof of Theorem \ref{Theorem explicit permutation cover} does not suppose the existence of covers and thus provides a proof of the existence of permutation covers independent of Theorem \ref{covers exist}.  

We now use the explicit description of the permutation cover to obtain a characterization of $R$-permutation modules.  Here and in future, if $U$ is an $RG$-module then $\overline{U}$ denotes the module $U/\pi U$, which may be treated as an $RG$-module or as an $(R/\pi)G$-module, depending on context.

\begin{prop}\label{Prop coflasque lifting to permutation}
Let $R$ be a complete discrete valuation ring whose residue class field $k = R/\pi$ has characteristic $p$.  If $X$ is a coflasque pseudocompact $RG$-lattice and $\overline{X} = X/\pi X$ is a $k$-permutation $kG$-module, then $X$ itself is an $R$-permutation $RG$-module.
\end{prop}

\begin{proof}
One may check using the explicit construction of the permutation cover $\theta_X:P_X\to X$ that its reduction $\overline{\theta_X}:\overline{P_X}\to \overline{X}$ modulo $\pi$ is the permutation cover of $\overline{X}$ using the following facts (cf.\ \cite{Broue}): 
\begin{itemize}
\item For any finite group $H$, reduction modulo $\pi$ induces a natural bijection between projective $RH$-modules and projective $kH$-modules;

\item Since $\pi X$ (being isomorphic to $X$) is coflasque, applying $(-)^P$ to the short exact sequence $0\to \pi X\to X\to \overline{X}\to 0$ yields the isomorphism
$$\overline{X}^P \iso X^P/(\pi X)^P = X^P/\pi X^P = \overline{X^P}.$$
Hence $X^{\db{P}} = \overline{X}^{\db{P}}$.  

\item Reduction modulo $\pi$ induces a natural bijection between Green correspondents of indecomposable $R$-permutation modules and those of the corresponding indecomposable $k$-permutation modules.

\end{itemize}
But $\overline{\theta_X}$ is an isomorphism by hypothesis.  So the kernel of $\theta_X$ is a pure submodule of $P_X$ contained in $\pi P_X$, and hence is $0$.  That is, $P_X\iso X$ and $X$ is a permutation module.
\end{proof}

\section{A theorem of Cliff and Weiss}\label{Section Cliff and Weiss}

From now on, $R$ will be a complete discrete valuation ring whose field of fractions has characteristic 0 and whose residue class field $k = R/\pi$ has prime characteristic $p$.  Given an $RG$-lattice $L$ we want to find a small power $\pi^e$ such that when $L/\pi^eL$ is a direct summand of a permutation module over $R/\pi^e$, then $L$ is coflasque.  We adapt the argument in \cite{Cliff and Weiss} to see that if some given $e$ has this property for every lattice for a cyclic group of order $p$, then this same $e$ will work for the whole group $G$.  Indeed, we may restrict to the case where $G$ is a $p$-group because $L\ds L\res{P}\ind{G}$ for any Sylow $p$-subgroup of $G$ so that, by the Eckmann-Shapiro Lemma and Mackey decomposition, we need only check that $L\res{P}$ is coflasque.  Suppose that $G$ is a $p$-group and proceed by induction on its order, the case of order $p$ being by hypothesis.  Let $Q$ be a central subgroup of order $p$ in $G$.  Since $H^1(Q,L)=0$ the sequence
$$0\to L^Q \xrightarrow{\pi^e} L^Q\to (L/\pi^eL)^Q\to 0$$
is exact, so that $L^Q/\pi^eL^Q \iso (L/\pi^eL)^Q$ is a permutation $(R/\pi^e)[L/Q]$-lattice.  By induction, $L^Q$ is coflasque as a module for $G/Q$.  Now, the inflation-restriction exact sequence (see \cite[6.8.3]{Weibel})
$$0\to H^1(G/Q,L^Q)\to H^1(G,L)\to H^1(Q,L)^{G/Q}$$
shows that $H^1(G,L)=0$, as required.

\begin{lemma}\label{cyclic group projective}
Let $G$ be a finite group and let $L$ be an $RG$-lattice in $RG\dash\tn{Mod}^{\mathcal{C}}$.  Suppose that $F$ is a free $kG$-submodule of $L/\pi L$.  Then there is a free $RG$-lattice $\widetilde{F}\leqslant L$ that is a summand of $L$ and such that $\widetilde{F}/\pi\widetilde{F} \iso (\widetilde{F}+\pi L)/\pi L = F$.
\end{lemma}

\begin{proof}
This is well known.  Let $\widehat{F}\to F$ be the projective cover of $F$ as an $RG$-module.  The composition $\widehat{F}\to F\to L/\pi L$ lifts to $\gamma:\widehat{F}\to L$ since $\widehat{F}$ is projective.  Let $\widetilde{F} = \gamma(\widehat{F})$.  Then $\widetilde{F}$ is pure over $R$ in $L$ and so is an $R$-summand of $L$.  The induction and coinduction functors being isomorphic for finite groups, and $\gamma$ being injective since it is injective modulo $\pi$, mean that $\widetilde{F}$ is injective relative to the the trivial subgroup of $G$ (see the discussion after Lemma \ref{Lemma db Q killed for non conjugate perm module}).  Thus, being an $RG$-submodule and an $R$-summand of $L$, it is an $RG$-summand.
\end{proof}

\noindent{\emph{Remark}:} Let $G$ be a $p$-group.  Replacing $F$ with a permutation module $V_H\ind{G}_H\leqslant L/\pi L$, let $\widehat{V}_H$ be the projective cover of $V_H$ as an $R$-module, treated as a trivial $RH$-module.  Then $\widehat{V}_H$ lifts to $\widetilde{V}_H\leqslant L$ if $H^1(H,L)=0$, and thus we get a submodule $\widetilde{V}_H\ind{G}_H\leqslant L$.  This is another way to prove Proposition \ref{Prop coflasque lifting to permutation}.

\medskip

Let $C$ be a group of order $p$ and let $L$ be an $RC$-lattice such that $L/\pi^eL$ is a permutation $(R/\pi^e)C$-lattice.  The module $L/\pi^eL$ decomposes as the sum of a free $(R/\pi^e)C$-lattice and a trivial $(R/\pi^e)C$-lattice.  By Lemma \ref{cyclic group projective} the free part lifts to a summand of $L$, whose complement we denote by $M$.  Note that $M/\pi^eM$ is trivial. 

Observe that a short exact sequence of lattices $J\to M\to X$ is necessarily split over $R$, so that the reduction $J/\pi^eJ \to M/\pi^eM\to X/\pi^eX$ is exact, hence if $M/\pi^eM$ is trivial then so are $J/\pi^eJ$ and $X/\pi^eX$.  The $RC$-lattice $M$ is an inverse limit of lattices of finite rank by Lemma \ref{RG lattice limit of finite rank RG lattices} and these quotients are necessarily trivial modulo $\pi^e$ by the above observation.  A lattice of finite rank is an iterated extension of irreducible lattices \cite[Cor.\ 23.15 and \S 25]{CR vol1}, where by irreducible lattice we mean a lattice $S$ such that $K\otimes_RS$ is irreducible as a $KC$-module, ($K$ the field of fractions of $R$).  Thus if we express each finite rank lattice as an iterated extension of irreducible lattices, then each of these must also be trivial modulo $\pi^e$.

The irreducible $KC$-modules are either trivial or of the form $K(\omega)$ for $\omega$ a primitive $p$th root of unity, where a given generator $c$ of $C$ acts by multiplication by $\omega^r$ with $r$ not divisible by $p$.  A lattice in such a module corresponds to a fractional ideal of $K(\omega)$ and such ideals are principal.  Thus, such a lattice is isomorphic to $R[\omega]$.

We have the following two cases:
\begin{itemize}
\item $\omega\not\in K$: Then the extension $K(\omega)/K$ is ramified and so $\omega-1\not\in \pi R[\omega]$.  In this case we can take $e=1$ so that the irreducible lattices are trivial.  But extensions and inverse limits of trivial lattices are trivial, so that our infinite lattice is also trivial.
\item $\omega\in K$, hence $\omega\in R$: Then $\omega-1\not\in \langle(\omega-1)\pi\rangle$, so the action of $C$ on $R[\omega]$ is not trivial modulo $(\omega-1)\pi$.  If we choose $e$ so that $\langle\pi^e\rangle = \langle(\omega-1)\pi\rangle$ then the infinite lattice must have been trivial.
\end{itemize}

Putting these observations together we obtain a form of a result of \cite{Cliff and Weiss}:

\begin{theorem}\label{ThmCliffWeiss}
Let $G$ be a finite group and let $L$ be a lattice in $RG\dash\tn{Mod}^{\mathcal{C}}$.
\begin{enumerate}
\item If $R$ does not contain a primitive $p$th root of unity then $L$ is an $R$-permutation lattice over $R$ if, and only if, $L/\pi L$ is a $k$-permutation module.
\item If $R$ contains a primitive $p$th root of unity $\omega$ then $L$ is an $R$-permutation lattice over $R$ if, and only if, $L/(\omega-1)\pi L$ is an $R/(\omega-1)\pi$-permutation lattice.
\end{enumerate}
\end{theorem}

Observe that, by considering lattices of rank 1, we already see that the power of $\pi$ in part 2 is the best possible.

\section{Weiss' Theorem}\label{Section Weiss}

Our goal for this section is to prove Theorem \ref{main}.  A special case of this result was proved in \cite{infgen} with $R=\Z_p$.  The missing link needed to obtain the result in generality is Theorem \ref{Weiss 3}, a generalization of \cite[Thm 3]{weiss} for infinitely generated lattices. Throughout this section, $G$ is a finite $p$-group and $R$ is a complete discrete valuation ring whose residue class field has characteristic $p$ and whose field of fractions has characteristic 0.  Denote by $\pi$ a generator of the maximal ideal of $R$ and let $k = R/\pi$.

If $U$ is an $RG$-module, recall that by $\overline{U}$ we denote the $kG$-module $U/\pi U$.  When $R$ contains a primitive $p$th  root of unity $\omega$, let $\widetilde{R} = R/(\omega-1)$ and denote by $\widetilde{U}$ the $\widetilde{R}G$-module $U/(\omega-1)U$.  For $S$ any quotient ring of $R$ (including $R$), a pseudocompact $SG$-lattice is said to be \emph{diagonal} if it is isomorphic to a direct product of indecomposable lattices of rank $1$ over $S$ (diagonal is thus a special case of monomial).

\begin{lemma}\label{Lemma diagonal sequence splits}
Suppose that $R$ contains a primitive $p$th root of unity $\omega$.  If a short exact sequence of pseudocompact $\widetilde{R}G$-modules $0\to X'\xrightarrow{s} X\xrightarrow{h}X''\to 0$ splits as $\widetilde{R}$-modules and if the middle term is diagonal, then the sequence splits as $\widetilde{R}G$-modules. 
\end{lemma}

\begin{proof}
We mimic the proof of \cite[Lem.\ 52.2]{weisspi}.  Write $X = \prod_I X_i$ with each $X_i$ an $\widetilde{R}G$-sublattice of $X$ of $\widetilde{R}$-rank 1.  The sequence remains split over $\overline{R}$ and the $\overline{X_i}$ are indecomposable $\overline{R}$-modules, so by Proposition \ref{prop exchange property} applied to these $\overline{R}$-modules there is a subset $J$ of $I$ such that $\overline{X} = \overline{s}(\overline{X'})\oplus \prod_{J}\overline{X_i}$, and this remains true over $\overline{R}G$.  Thus the $\overline{R}G$-module homomorphism $\overline{h}' = \overline{h}|_{\prod_J \overline{X_i}}$ is an isomorphism, where $\overline{h}$ is the map induced by $h$.  By Nakayama's Lemma, $h|_{\prod_J X_i} = h':\prod_J X_i \to X''$ is an isomorphism.  Now the composition of $h'^{-1}$ with the inclusion $\prod_J X_i \to X$ is a splitting of $h$.
\end{proof}

In the above proof we work modulo $\pi$ because there need not be finitely many isomorphism classes of rank $1$ $\widetilde{R}G$-lattice, so that Proposition \ref{prop exchange property} need not apply directly to $X$.

\begin{prop}\label{Prop Weiss 3 diagonal quotient}
Suppose that $R$ contains a primitive $p$th root of unity $\omega$.  Let $U$ be a pseudocompact $RG$-lattice and $V$ a diagonal pseudocompact $RG$-lattice.  If $\widetilde{U}\iso \widetilde{V}$, then $U$ is diagonal.
\end{prop}

\begin{proof}
Write $U = \invlim_i U/X_i$ with each $U/X_i$ an $RG$-lattice of finite rank over $R$, as we may by Lemma \ref{RG lattice limit of finite rank RG lattices}.  As there are only finitely many isomorphism classes of indecomposable diagonal $RG$-lattices, by Theorem \ref{thm limits of Add are in Add} applied with $M$ the direct sum of a representative of each isomorphism class of indecomposable diagonal $RG$-lattice, it suffices to check that each $U/X_i$ is diagonal.  The sequence $0\to X_i\to U \to U/X_i\to 0$ is split over $R$ because $U/X_i$ is a lattice, so applying $\widetilde{(-)}$ it follows that $\widetilde{U/X_i}\iso \widetilde{U}/\widetilde{X_i}$.   Denoting by $\rho: \widetilde{U}\to \widetilde{V}$ the given isomorphism, we have
$$\widetilde{U/X_i}\iso \widetilde{U}/\widetilde{X_i} \iso \widetilde{V}/\rho(\widetilde{X_i}).$$
By Lemma \ref{Lemma diagonal sequence splits}, $\widetilde{V}/\rho(\widetilde{X_i})$ is a direct summand of $\widetilde{V}$, hence diagonal by the Krull-Schmidt property.  It follows that the finitely generated $RG$-lattice $U/X_i$ is diagonal modulo $\omega-1$, and hence is diagonal by the finitely generated version of this result \cite[Thm 50.2]{weisspi} (in fact, we only require Case B in the proof).  Note that the cited proof is for compact discrete valuation rings, but the same proof holds for complete discrete valuation rings because we still have Nakayama's Lemma.
\end{proof}

\begin{lemma}\label{injection mod pi means split injection}
Suppose that $f : L\to M$ is a (continuous) homomorphism of pseudocompact $R$-lattices such that $\overline{f} : \overline{L}\to \overline{M}$ is injective.  Then $f$ is a split injection in $R\dash\tn{Mod}^{\mathcal{C}}$.
\end{lemma}

\begin{proof}
If $f$ is not injective, then there exists $x\neq 0$ in the kernel.  Write $x=\pi^ny$ for some $y\not\in \pi L$.  Now
$$0 = f(x) = f(\pi^ny) = \pi^nf(y),$$
so that $f(y)=0$.  Hence there exists $y\not\in \pi L$ in the kernel of $f$.  But now
$$\overline{f}(\overline{y}) = \overline{f(y)} = 0,$$
contradicting the injectivity of $\overline{f}$.

\medskip

We claim that $f(L)$ is pure in $M$.  If not, then there is $y\in L\setminus \pi L$ such that $f(y) = \pi m$ for some $m\in M$.  Now $\overline{y}\neq 0$, while
$$\overline{f}(\overline{y}) =\overline{f(y)} = 0,$$
contradicting the injectivity of $\overline{f}$.  Thus $M/f(L)$ is torsion-free, hence free, and so $f$ splits.
\end{proof}

We next generalize \cite[Prop.\ 53.2]{weisspi} to infinite rank lattices:

\begin{lemma}\label{Weiss 6}
Suppose we have a short exact sequence of $RG$-lattices
$$0\to L\to T \xrightarrow{\gamma} M \to 0$$
with $L,M$ monomial and such that the induced sequence
$$0\to \widetilde{L} \to \widetilde{T} \to \widetilde{M} \to 0$$
is split exact.  Then the sequence splits.
\end{lemma}

\begin{proof}
We first show that $\gamma$ is monomial supersurjective.  Writing $(\varphi\ind{G}_H, -)$ instead of $\tn{Hom}_{RG}(\varphi\ind{G}_H,-)$ and applying it to the given sequence, we obtain 
$$\xymatrix{
0\ar[r] & (\varphi\ind{G}_H,L)\ar[r] & (\varphi\ind{G}_H,T)\ar[r] & (\varphi\ind{G}_H,M) \ar[r]^{\partial\quad}\ar[d] & \tn{Ext}^1_{RG}(\varphi\ind{G}_H,L)\ar[d] \\
        &                            &                            & (\widetilde{\varphi\ind{G}_H},\widetilde{M}) \ar[r]^{0\quad} & \tn{Ext}^1_{RG}(\widetilde{\varphi\ind{G}_H},\widetilde{L}) \\
}$$
Write $L = \prod L_i$ as a product of indecomposable monomial modules.  We have 
$$\tn{Ext}^1_{RG}(\varphi\ind{G}_H,L) = \prod \tn{Ext}^1_{RG}(\varphi\ind{G}_H,L_i).$$
By \cite[Prop.\ 53.2]{weisspi}, on each factor of this product the vertical right map is injective, and hence the vertical right map is itself injective.  But this forces $\partial$ to be 0, showing that $ (\varphi\ind{G}_H,T) \to (\varphi\ind{G}_H,M)$ is surjective, so that $\gamma$ is monomial supersurjective.

By Lemma \ref{monomial precovers monomial supersurjective}, the monomial supersurjectivity of $\gamma$ implies that $(M,\gamma):(M,T)\to (M,M)$ is surjective.  An element of the preimage of $\tn{id}_M$ is the required splitting of $\gamma$.
\end{proof}

We need the following standard lemma (see after Lemma \ref{Lemma db Q killed for non conjugate perm module} for definitions).

\begin{lemma}\label{rel proj lemma}
Let $U$ be a finite direct sum of modules $U_i$ each projective relative to a proper subgroup of $G$.  Suppose that the epimorphism $\gamma:Y\onto U$ splits over every proper subgroup of $G$.  Then $\gamma$ splits.
\end{lemma}

\begin{proof}
Suppose that each $U_i$ is relatively $H_i$-projective for $H_i<G$.  Then $\tn{id}_{U_i} = \tn{Tr}_{H_i}^G(\alpha_i)$ for some $\alpha_i\in \tn{End}_{RH_i}(U_i)$.  Hence $\tn{id}_U = \sum_i \tn{Tr}_{H_i}^G(\alpha_i)$.  For each $i$, let $s_i$ be a splitting of $\gamma$ as an $H_i$-module homomorphism.  Using \cite[Lem.\ 3.6.3]{Benson} we obtain
$$\gamma\left(\sum_i\tn{Tr}_{H_i}^G(s_i\alpha_i)\right) = \sum_i\tn{Tr}_{H_i}^G(\gamma s_i\alpha_i) = \sum_i \tn{Tr}_{H_i}^G(\alpha_i) = \tn{id}_U.$$ 
\end{proof}

The next theorem is a generalization of \cite[Thm 50.2]{weisspi} to pseudocompact modules.

\begin{theorem}\label{Weiss 3}
Let $R$ be a complete discrete valuation ring containing a primitive $p$th root of unity $\omega$.  If $U$ is a pseudocompact lattice such that $\widetilde{U}$ is monomial, then $U$ itself is monomial.
\end{theorem}

\begin{proof}
We work by induction on the order of $G$.  If $|G|=1$ the result is obvious.  We suppose that the result holds for any proper subgroup of $G$.

Applying the functor $\widetilde{(-)}$ to the monomial cover $c:P\to U$ yields the commutative square
$$\xymatrix{
P\ar[r]^{c}\ar[d] & U\ar[d] \\
\widetilde{P}\ar[r]_{\widetilde{c}} & \widetilde{U}}$$
For any proper subgroup $H$ of $G$ the homomorphism $c\res{H}$ is a surjection from the monomial module $P\res{H}$ to the module $U\res{H}$, the latter monomial by induction.  Using Frobenius reciprocity, we see that $c\res{H}$ is a monomial precover.  Hence, since $U\res{H}$ is monomial, the homomorphism $c\res{H}$ is a split surjection for every proper subgroup $H$ of $G$.  Thus $\widetilde{c}$ also splits over every proper subgroup of $G$.

The module $\widetilde{U}$ is monomial by hypothesis, so write it as $T\oplus Y$, where $T$ is diagonal and $Y$ is projective relative to proper subgroups of $G$.  Consider the composition
$$\widetilde{P} \xonto{\widetilde{c}} T\oplus Y \xonto{\pi_Y}Y.$$
This homomorphism splits over every proper subgroup of $G$ (because $\widetilde{c}$ does) and hence, since $Y$ is projective relative to proper subgroups, it splits over $G$  (to see this, apply Lemma \ref{rel proj lemma} to $Y$ written as a finite sum of modules each projective relative to a fixed maximal subgroup of $G$, for example).  Let $\beta: Y\to \widetilde{P}$ be an $\widetilde{R}G$-module homomorphism such that $\pi_Y\widetilde{c}\beta = \tn{id}_Y$.  Define the following submodules of $\widetilde{P}$:
$$Y' = \beta(Y),$$
$$T' = \widetilde{c}^{-1}(T).$$
Routine checks show that $\widetilde{P} = Y'\oplus T'$.
Observe that $\widetilde{c}|_{Y'}: Y'\to Y$ is an isomorphism.

\medskip

Write $P = \prod_{i\in I}P_i$ with the $P_i$ indecomposable monomial lattices.  Hence $\widetilde{P} = \prod_I \widetilde{P_i}$.  This module has the exchange property (Proposition \ref{prop exchange property} can be applied because there are finitely many isomorphism classes of indecomposable monomial $RG$-lattice, hence finitely many $\widetilde{P_i}$ up to isomorphism), so we can write 
$$\widetilde{P} = \prod_{i\in J} \widetilde{P_i} \oplus T'$$
for some subset $J\subseteq I$.

Let $W = \prod_J P_i$ be the corresponding summand of $P$.  It follows that $\widetilde{P} = \widetilde{W}\oplus T'$, and hence that
$$\widetilde{W} \iso \widetilde{P}/T'\iso Y'.$$

The reader may check that $\widetilde{c}|_{\widetilde{W}}$ is injective. 
We claim that $\widetilde{U} = \widetilde{c}(\widetilde{W})\oplus T$: That $\widetilde{U} = \widetilde{c}(\widetilde{W})+T$ is obvious, so consider $u\in \widetilde{c}(\widetilde{W})\cap T$, and write $u=\widetilde{c}(w)$ for $w\in \widetilde{W}$.  Then 
$$w\in \widetilde{W} \cap \widetilde{c}^{-1}(T) = \widetilde{W}\cap T' = 0,$$
so that $u = \widetilde{c}(0) = 0$.

We have the following diagram:
$$\xymatrix{
0\ar[r] & W\ar[r]^{c|_{W}}\ar[d] &  U\ar[r]\ar[d] & U/W\ar[r]\ar[d]   & 0 \\ 
0\ar[r] & \widetilde{W}\ar[r]            & \widetilde{U}=T\oplus Y\ar[r]    & \widetilde{U/W}\ar[r] & 0
}$$
The second row is exact because $\widetilde{(-)}$ is right exact and $\widetilde{c}|_{\widetilde{W}}$ is injective.  It follows from the claim above that $\widetilde{U/W} \iso T$.  Thus, by Proposition \ref{Prop Weiss 3 diagonal quotient}, $U/W$ is monomial.  Meanwhile, $W$ is monomial by construction.

Finally, the lower sequence splits ($\widetilde{c}|_{\widetilde{W}}$ is a split injection) so by Lemma \ref{Weiss 6}, $U\iso W\oplus U/W$ is monomial, as required.
\end{proof}

We prove Theorem \ref{main}.  We work by induction on the order of $G$.  If $N=1$ the result is immediate so suppose that $N\neq 1$.  There exists a central normal subgroup $M$ of $G$ of order $p$ contained in $N$.  Consider the $G/M$-module $U^M$.  We have $(U^M)^N = U^N$ is a permutation $G/N = (G/M)/(N/M)$-module, and $U^M\res{N/M}\iso (U\res{N})^M$ is a free $N/M$-module.  Hence by induction we have that $U^M$ is a permutation $G/M$-lattice.  But we also have that $U\res{M} = U\res{N}\res{M}$ is free, and hence we need only show that the theorem holds when $N$ is central of order $p$.  We assume this from now on.  Denote by $\Sigma_N$ the element $\sum_{n\in N}n\in RN$.

\begin{lemma}\label{kernel SigmaN implies iso}
Let $N$ be a central subgroup of $G$ of order $p$ and let $U$ be a pseudocompact $RG$-lattice such that $U^N$ is a permutation module and $U\res{N}$ is free.  Let $L$ be a pseudocompact permutation $RG$-lattice and let $f: L\to U/\Sigma_NU$ be a surjection having kernel $\Sigma_NL$.  Then $L\iso U$.
\end{lemma}

\begin{proof}
Note that $RN/RN^N$ is torsion-free (as an $R$-module), so $U/U^N$ is torsion-free since $U\res{N}$ is free.  Note also that since $U$ is free over $N$, we have $\Sigma_N U = U^N$.  Write $L = \prod_i R[G/H_i]$ (via Proposition \ref{Prop perm module characterization}).  If for some $i$ we had $N\leqslant H_i$, then by the Mackey decomposition, $L\res{N}$ would have a trivial summand $R$.  But then $R/\Sigma_NR = R/pR$ would be isomorphic to a submodule of $U/\Sigma_NU$ -- impossible, since $U/\Sigma_NU$ is torsion-free.  Hence $L\res{N}$ is free.

Because $\Sigma_N U = U^N$ is a permutation module by hypothesis, we have $H^1(H,\Sigma_NU)=0$ for all $H$ by Lemma \ref{perm modules coflasque}.  Thus the homomorphism $U\onto U/\Sigma_N U$ is supersurjective and hence, by Lemma \ref{perm precovers supersurjective}, $f$ lifts to a homomorphism $\widehat{f}:L\to U$.

Consider the triangle
$$\xymatrix{
   &   L\ar[dl]_{\widehat{f}}\ar[dr]^f  &   \\
U\ar[rr]  &  & U/\Sigma_NU}$$
Considering this triangle modulo $\Sigma_N$ shows that $\widehat{f}$ is an isomorphism modulo $\Sigma_N$.  But $RN$ is local, hence $\widehat{f}$ is surjective by Nakayama's Lemma and injective by Lemma \ref{injection mod pi means split injection}. 
\end{proof}

By the discussion before Lemma \ref{kernel SigmaN implies iso}, Theorem \ref{main} follows from the following special case:

\begin{theorem}
Let $R$ be a discrete valuation ring whose field of fractions has characteristic 0 and whose residue field has characteristic $p>0$.  Let $G$ be a finite $p$-group and let $U$ be a pseudocompact $RG$-lattice.  Suppose $G$ has a central subgroup $N$ of order $p$ such that
\begin{itemize}
\item $U\res{N}$ is free,
\item $U^N$ is a permutation module.
\end{itemize}
Then $U$ itself is a permutation module.
\end{theorem}

\begin{proof}
We prove this first in the special case that $p$ is prime in $R$.  Adjoin a primitive $p$th root of unity $\omega$ to $R$ to obtain the ring $R[\omega] = R[x]/(1+x+\hdots+x^{p-1})$.  The kernel of the surjective ring homomorphism $R[\omega]\to R/p$ sending $\omega$ to $1$ is generated by $\omega-1$, so that $p$ prime in $R$ implies that $\omega-1$ is prime in $R[\omega]$.  If $p=2$ then $R[\omega] = R$, but this does not affect the argument.  
Fix an isomorphism of groups $\psi:N\to \langle\omega\rangle$. 
If $p$ is odd, define the structure of $R[\omega]G$-lattice on the module $U/\Sigma_NU$ via $\psi$ by defining the action of $\omega$ as
$$\omega\cdot (u+\Sigma_NU) := \psi^{-1}(\omega)u + \Sigma_NU,$$
(if $p$ is $2$ there is nothing to define, but in this case note that the formula agrees with the action of $\omega=-1$ on $U/\Sigma_NU$). This structure is natural in the sense that having fixed $\psi$, we can consider $U\mapsto U/\Sigma_NU$ as a functor $RG\dash\tn{Mod}\to R[\omega]G\dash\tn{Mod}$. 
Multiplication by $\Sigma_N$ yields an isomorphism
$$(U/\Sigma_NU)/(\omega-1)(U/\Sigma_NU) \iso U^N/pU^N.$$
This is a permutation module (because $U^N$ is), and hence by Theorem \ref{Weiss 3}, $U/\Sigma_NU$ is a monomial $R[\omega]G$-lattice.  So write
$$U/\Sigma_NU = \prod_{i\in I}V_i\ind{G}_{H_i},$$
where $V_i$ is an $R[\omega]H_i$ lattice of rank 1.  As $R[\omega]$ contains no primitive $p^2$ root of unity, the action of each $H_i$ on the corresponding $V_i$ is given by a group homomorphism $\varphi_i$ from $H_i$ to $\langle\omega\rangle$.  Denote by $K_i$ the kernel of $\varphi_i$.  The $RN$-module $RN/\Sigma_N RN$ has no non-zero $N$-fixed points and hence, since $U$ is $N$-free, $U/\Sigma_NU$ has no non-zero $N$-fixed points.  Calculating the $N$-fixed points of $V_i\ind{G}_{H_i}$ differently, we obtain isomorphisms of $R$-modules
\begin{align*}
    0 = \left(V_i\ind{G}_{H_i}\res{N}\right)^N 
& \iso \bigoplus_{g\in G/H_iN}\tn{Hom}_N(R, {}^g\!V_i\res{{}^g\!H_i\cap N}\ind{N}) \\
& \iso \bigoplus_{g\in G/H_iN}\tn{Hom}_{{}^g\!H_i\cap N}(R, {}^g\!V_i\res{{}^g\!H_i\cap N}).
\end{align*}
It follows that $H_i\cap N$ is not contained in $K_i$ and hence that $N$ is a subgroup of $H_i$ not intersecting $K_i$.  In particular, $\varphi_i$ is not trivial.

A surjective $RH_i$-module homomorphism $R[H_i/K_i]\to V_i/(\omega-1)V_i\iso R/p$ lifts by projectivity over $H_i/K_i$ to a surjective $RH_i$-module homomorphism from $R[H_i/K_i]$ to $V_i$.
Inducing to $G$, we obtain a surjective $RG$-module homomorphism $R[G/K_i]\to V_i\ind{G}_{H_i}$ whose kernel contains (hence is equal to by comparison of ranks) $\Sigma_NR[G/K_i]$, so we have an isomorphism of $RG$-modules
$$R[G/K_i]/\Sigma_NR[G/K_i] \iso V_i\ind{G}_{H_i}.$$
Define the permutation lattice $L = \prod_{i\in I} R[G/K_i]$ together with the obvious homomorphism $L\to U/\Sigma_N U$, whose kernel is $\Sigma_N L$.  Now by Lemma \ref{kernel SigmaN implies iso}, $U\iso L$ is a permutation lattice.  This completes the special case.

\medskip

For the general case, let $S$ be a coefficient ring for $R$, in the sense of Cohen's Structure Theorem \cite{Cohen}.  Thus $S$ is a complete discrete valuation ring contained in $R$ with the following properties:
\begin{itemize}
\item The inclusion $S\to R$ realizes $R$ as a free pseudocompact $S$-module of finite rank over $S$ (cf.\cite[Ch.\ 1 \S 6]{SerreLF}). 

\item The number $p$ generates the maximal ideal of $S$.
\end{itemize}

Denote by $U'$ the module $U$ considered as an $SG$-module. Then $U'$ satisfies the hypotheses of the theorem, so that $U'$ is a permutation $SG$-module by the special case above.  As $U'$ is free over $N$, it is a product of modules $S[G/H]$ with $H\cap N=1$.  Since $\tn{id}:U'\to U'$ is the permutation cover of $U'$, we have that $R\ctens_S U'\to U$ given by $r\ctens u\mapsto ru$ is supersurjective, so is a permutation precover of $U$ by Lemma \ref{perm precovers supersurjective}.  It follows that the permutation cover of $U$ (a direct summand of $R\ctens_S U'$) is free on restriction to $N$.  Denote by 
$$0\to K\to C\xrightarrow{f} U\to 0$$
the permutation cover of $U$.  We will show that $K=0$.  By freeness of $U$ over $N$ the sequence is split over $N$ and hence
$$0\to K_N\to C_N\xrightarrow{f_N}U_N\to 0$$
is exact.  Note that $U_N$ is a permutation module and $f_N$ is supersurjective, both these claims following from the natural isomorphisms $C_N\to C^N$ and $U_N\to U^N$ given by multiplication by $\Sigma_N$, again by freeness of $C$ and $U$ over $N$.  Thus $f_N$ is a permutation precover of the permutation module $U_N$, so it splits and hence $K_N$ is a direct summand of $C_N$.  Write $C = \prod_I V_i$, a product of indecomposable modules.  Then each $(V_i)_N$ is indecomposable, and $C_N = \prod_I (V_i)_N$. By the extension property, there is a subset $J$ of $I$ such that $C_N = K_N\oplus \prod_J (V_i)_N$.  By Nakayama's Lemma, $C = K + \prod_JV_i$.  The restriction $f:\prod_J V_i \to U$ is thus also a precover of $U$, so that $I=J$ since $f$ is a cover.  It follows that $K_N=0$, hence $K=0$ as required.
\end{proof}

\end{document}